\documentclass{amsart}
\usepackage[utf8]{inputenc}
\usepackage{amsmath, amssymb, amsthm, verbatim, hyperref}
\usepackage{mathtools}
\usepackage[dvipsnames]{xcolor}
\usepackage{ragged2e}
\usepackage{marginnote}
\usepackage{enumerate}
\usepackage{makecell}
\usepackage{longtable}
\usepackage{epsfig}
\usepackage{tikz, graphicx} 
\usepackage{tikz-cd}
\usepackage[all,arc]{xy}
\usepackage{ytableau}
\usepackage{hyperref}
\usepackage{caption}
\usepackage{subcaption}
\hypersetup{hidelinks}

\newcommand{\linext}{{\mathcal{L}}}
\newcommand{\Stab}{{\operatorname{Stab}}}
\newcommand{\CST}{{\operatorname{CST}}}
\newcommand{\SYT}{{\operatorname{SYT}}}
\newcommand{\symm}{{\mathfrak{S}}}
\newcommand{\ang}[1]{\langle #1 \rangle}

\newcommand{\bkp}{\mathcal{BK}_P}
\newcommand{\bkq}{\mathcal{BK}_Q}

\makeatletter
\DeclareRobustCommand{\rvdots}{%
  \vbox{
    \baselineskip4\p@\lineskiplimit\z@
    \kern-\p@
    \hbox{.}\hbox{.}\hbox{.}
  }}
\makeatother

\theoremstyle{plain}

\newtheorem{thm}{Theorem}[section]
\newtheorem{prop}[thm]{Proposition}
\newtheorem{cor}[thm]{Corollary}
\newtheorem{lemma}[thm]{Lemma}
\newtheorem{conjecture}[thm]{Conjecture}

\theoremstyle{definition}
\newtheorem{definition}[thm]{Definition}
\theoremstyle{remark}

\newtheorem{question}[thm]{Question}
\newtheorem{remark}[thm]{Remark}

\newtheorem{example}[thm]{Example}

\title[Bender--Knuth involutions on linear extensions of posets]{Bender--Knuth involutions\\ on linear extensions of posets}

\author[Chiang]{Judy Hsin-Hui Chiang}
\address{Department of Mathematics, University of Illinois, Urbana, IL} \email{hsinhui2@illinois.edu}
\thanks{JC -- Present address: School of Mathematics, University of Minnesota, Minneapolis, MN. Email: \href{mailto:hchiang@umn.edu}{\it hchiang@umn.edu}}

\author[Hoang]{Anh Trong Nam Hoang}
\address{School of Mathematics, University of Minnesota, Minneapolis, MN}
\email{hoang278@umn.edu}

\author[Kendall]{Matthew Kendall}
\address{Department of Mathematics, Princeton University, Princeton, NJ}
\email{mskendall@princeton.edu}

\author[Lynch]{Ryan Lynch}
\address{Department of Mathematics, University of Notre Dame, Notre Dame, IN} \email{rlynch6@nd.edu}
\thanks{RL -- Present address: School of Mathematics, University of Minnesota, Minneapolis, MN. Email: \href{mailto:rlynch@umn.edu}{\it rlynch@umn.edu}}

\author[Nguyen]{Son Nguyen}
\address{School of Mathematics, University of Minnesota, Minneapolis, MN}
\email{nguy4309@umn.edu}

\author[Przybocki]{Benjamin Przybocki}
\address{Department of Mathematics, Stanford University, Stanford, CA}
\email{benprz@stanford.edu}

\author[Xia]{Janabel Xia}
\address{Department of Mathematics, Massachusetts Institute of Technology, Cambridge, MA}
\email{janabel@mit.edu}

\date{\today}

\begin{document}
\ytableausetup{centertableaux}

\begin{abstract}
    We study the permutation group $\bkp$ generated by Bender--Knuth moves on linear extensions of a poset $P$, an analog of the Berenstein--Kirillov group on column-strict tableaux. We explore the group relations, with an emphasis on identifying posets $P$ for which the cactus relations hold in $\bkp$. We also examine $\bkp$ as a subgroup of the symmetric group $\symm_{\linext(P)}$ on the set of linear extensions of $P$ with the focus on analyzing posets $P$ for which $\bkp = \symm_{\linext(P)}$.
\end{abstract}

\maketitle

\tableofcontents


\section{Introduction}\label{sec:intro}

First introduced by Bender and Knuth \cite{benderknuth1972} in their study of enumerations of plane partitions and Schur polynomials, the {\it Bender--Knuth (BK) moves}, a certain family of involutions on the set of column-strict (semi-standard) Young tableaux, have seen a wide range of applications across different areas of combinatorics. They were shown to be equivalent to tableau switching, an involution on pairs of column-strict tableaux, on horizontal border strips of two adjacent letters, together with a swap of these labels \cite{bss96,pak-vallejo}. Berenstein and Kirillov \cite{berenstein-kirilov} studied an extension of these involutions, called piecewise-linear BK moves, acting on Gelfand--Tsetlin patterns and the relations they satisfy. Other interpretations and applications of BK moves appeared in the context of crystals for finite-dimensional complex reductive Lie
algebras (e.g., \cite{halacheva}), shifted tableaux (e.g., \cite{stembridge,rodrigues}), and Grothendieck polynomials (e.g., \cite{ikeda-shimazaki,galashin-grinberg-liu}).

Informally, the BK moves $t_i$ act on a column-strict tableau by fixing an $i$ (resp. $i+1$) when there is an $i+1$ below (resp. $i$ above), and swapping the contents of the remaining letters $i$ and $i+1$ in each row. The {\it (combinatorial) Berenstein--Kirillov (BK) group} $\mathcal{BK}$ is defined to be the free group generated by variables $t_i$ for all $i \in \mathbb{Z}_{> 0}$, modulo the relations satisfied by the BK moves $t_i$ when acting on all column-strict tableaux. This group was formally introduced by Berenstein and Kirillov \cite{berenstein-kirilov}, whose work also included a list of relations in this group. More recently, Chmutov, Glick, and Pylyavskyy \cite{CGP} related the BK moves to the {\it cactus group} $\mathcal{C}_n$, the fundamental group of the moduli space of marked real genus zero stable curves (see, e.g., \cite{devadoss,davis-januszkiewicz-scott, henriques-kamnitzer,halacheva}). They showed that the BK moves $t_i$ acting on all column-strict tableaux satisfy the defining relations of $\mathcal{C}_n$, giving a group homomorphism from $\mathcal{C}_n$ to the subgroup $\mathcal{BK}_n \subset \mathcal{BK}$ generated by $\{ t_1, \ldots, t_{n-1} \}$ and yielding new relations previously unknown in $\mathcal{BK}$. The same subject was concurrently investigated by Berenstein and Kirillov \cite{berenstein-kirillov-RIMS} using a purely group-theoretic approach. These results serve as one of the motivations for our work.

Recall that a {\it linear extension} of a poset $P$ is a linear order that is compatible with $P$. Haiman \cite{haiman} and Malvenuto and Reutenauer \cite{malvenuto-reutenauer} introduced an analog of the BK moves $t_i$ on linear extensions of a poset $P$, which swap two adjacent letters $i$ and $i+1$ when they label incomparable elements of $P$ and fix them otherwise, and used them to study {\it promotion} and {\it evacuation}, operators on linear extensions first defined by Sch\"utzenberger \cite{schutzenberger1976,schutzenberger77}. A survey on basic properties and generalizations of these operators can be found in \cite{Stanley-promotion-evacuation}.

An important tool to study BK moves on linear extensions is the {\it linear extension graph} of a poset $P$, the graph whose vertices are labelled by linear extensions of $P$ and edges are given by the BK moves that swap corresponding linear extensions. Linear extension graphs were first introduced by Pruesse and Ruskey \cite{pr1991}, and previously used in the study of linear extension generation (see, e.g., \cite{RUSKEY1992,stachowiak1992,west1993,naatz2000,bm2013}) as well as Markov chains on the set of linear extensions (e.g., \cite{ayyer-klee-schilling}). A survey on linear extension graphs can be found in \cite{MassowThesis}.

In this context, we may define an analog of the Berenstein--Kirillov group on linear extensions of a given poset $P$, denoted by $\bkp$. The goal of this paper is to study properties of this group $\bkp$, attempting to characterize the classes of posets $P$ for which $\bkp$ enjoys various properties.


\subsection{Outline of the paper}\label{subsec:outline}

Section~\ref{sec:basics} provides basic definitions and constructions that are fundamental to this paper, including column-strict tableaux, linear extensions of posets, and BK moves on various combinatorial objects. Given a poset $P$, the {\it Berenstein--Kirillov group of $P$}, $\bkp$, is defined to be the permutation group generated by the BK moves $t_i$ on linear extensions of $P$.

In Section~\ref{sec:rel}, we focus on the relations in $\bkp$. In particular, we will identify the posets $P$ for which the trivialization relations (Proposition~\ref{prop:rel_ti=1}) and braid relations (Proposition~\ref{prop:rel_braid}) hold; note that the latter was previously observed by \cite{ayyer-klee-schilling}.
Progress is made towards understanding the posets $P$ for which the {\it cactus relations}, fundamental relations in cactus groups, hold in $\bkp$.

Finally, Sections~\ref{sec:symm} and~\ref{sec:stab} are dedicated to the study of the group $\bkp$ as a permutation group on the set $\linext(P)$ of all linear extensions of $P$. Section~\ref{sec:symm} focuses on understanding posets $P$ for which $\bkp$ equals the full symmetric group on $\linext(P)$. More specifically, we classify all disconnected posets (Theorem~\ref{thm:symm_disconn}) and series-parallel posets (Corollary~\ref{cor:symm_seri_paral}), and exhibit a few other families of connected posets with this property. Section~\ref{sec:stab} explores the cardinality of $\bkp$.


\subsection{Acknowledgments}

We are indebted to Vic Reiner for initiating the subject of study and providing extremely valuable guidance. We would like to thank Gregg Musiker and the 2022 University of Minnesota Combinatorics and Algebra REU staff for organizing the program. Finally, we thank Pasha Pylyavskyy, Joel Kamnitzer, Iva Halacheva, Sylvester Zhang for helpful discussions, and the referee for constructive feedback. This research was partially supported by RTG grant NSF/DMS-1745638.


\section{Basic definitions and constructions}\label{sec:basics}

In this section, we will provide the background on several algebraic and combinatorial objects that are fundamental to this paper, including column-strict and standard Young tableaux, linear extensions, Bender--Knuth involutions, and poset operations.


\subsection{Column-strict tableaux and Bender--Knuth moves}\label{subsec:BK_CST}

Recall that a {\it partition} $\lambda$ of a positive integer $N$ is a tuple of integers $\lambda = (\lambda_1, \lambda_2, \ldots, \lambda_n)$ such that $\lambda_1 \ge \lambda_2 \ge \ldots \ge \lambda_n \ge 0$ and $\sum_{i=1}^n \lambda_i = N$. The {\it Young diagram of shape} $\lambda$ is a finite collection of boxes arranged in left-justified rows, where the $i^{\mathrm{th}}$ row has $\lambda_i$ boxes. A {\it column-strict (semi-standard) tableau} $T$ of shape $\lambda$ is a filling of the Young diagram of shape $\lambda$ that is weakly increasing along the  rows and strictly increasing along the columns. A tableau $T$ is said to have {\it content} $\alpha = (\alpha_1, \alpha_2, \ldots,\alpha_n)$ if there are $\alpha_i$ occurrences of $i$ in $T$ for each $i=1,\ldots,n$; when the content is $\alpha = (1,1,\ldots,1)$, the tableau $T$ is called a {\it standard Young tableau}. Denote the set of all column-strict tableaux of shape $\lambda$ and content $\alpha$ by $\CST(\lambda, \alpha)$, and the set of standard Young tableaux of shape $\lambda$ by $\SYT(\lambda)$.

\begin{example}
The Young diagram of shape $(3,2)$ is
\[(3,2) = \;
    \ydiagram{3,2} \; .
\]
The collection of column-strict tableaux of shape $(3,2)$ and content $(1,2,1,1)$ is
\[
    \CST((3,2),(1,2,1,1)) = 
    \left\{ 
    \begin{footnotesize}
    \; \begin{ytableau} 
    1 & 2 & 2  \\ 
    3  & 4 & \none 
    \end{ytableau} \; , \;
    \begin{ytableau}
      1 & 2 & 3  \\
      2  & 4 & \none
    \end{ytableau} \; , \;
    \begin{ytableau}
      1 & 2 & 4  \\
      2  & 3 & \none
    \end{ytableau} \;
    \end{footnotesize}
    \right\}.
\]
Meanwhile, that of standard Young tableaux of the same shape is
\[
    \SYT((3,2)) = 
    \left\{ 
    \begin{footnotesize}
    \; \begin{ytableau} 
    1 & 2 & 3  \\ 
    4 & 5 & \none 
    \end{ytableau} \; , \;
    \begin{ytableau}
    1 & 2 & 4  \\
    3 & 5 & \none
    \end{ytableau} \; , \;
    \begin{ytableau}
    1 & 2 & 5  \\
    3 & 4 & \none
    \end{ytableau} \; , \;
    \begin{ytableau}
    1 & 3 & 4  \\
    2 & 5 & \none
    \end{ytableau} \; , \;
    \begin{ytableau}
    1 & 3 & 5  \\
    2 & 4 & \none
    \end{ytableau} \;
    \end{footnotesize}
    \right\}.
\]
\end{example}

The primary combinatorial object of interest in this paper is the {\it Bender--Knuth moves} (or {\it BK moves} for short), which were originally defined on column-strict tableaux by Bender and Knuth \cite{benderknuth1972} and further studied by several authors, e.g., \cite{SCHUTZENBERGER1972,berenstein-zelevenski,berenstein-kirilov,CGP}. Fix a partition $\lambda$, and let $T$ be a column-strict tableau of shape $\lambda$ and content $\alpha = (\alpha_1, \ldots, \alpha_n)$.

\begin{definition}\label{defn:BK_CST}
The {\it Bender--Knuth move} $t_i$ for $1 \le i \le n-1$ is an involution (i.e., a permutation of order $2$) of the set of column-strict tableaux of shape $\lambda$ that sends $T$ to the tableau obtained from the following procedure:
\begin{enumerate}
    \item Let $S$ be the skew tableau obtained by taking only the boxes of $T$ with entry equal $i$ and $i+1$;
    \item Observe that each row of $S$ contains
    \begin{enumerate}
    \item $a$ entries equal to $i$ that lie directly above an $i+1$,
    \item $b$ entries equal to $i$ that are alone in their columns,
    \item $c$ entries equal to $i+1$ that are alone in their columns, and
    \item $d$ entries equal to $i+1$ that lie directly below an $i$
    \end{enumerate}
    for some $a,b,c,d \ge 0$;
    \item Construct a skew tableau $S'$ by swapping $b$ and $c$ in each row of $S$;
    \item Define $t_i(T)$ to be the tableau obtained by replacing $S$ with $S'$ in $T$.
\end{enumerate}
\end{definition}

\begin{example}
Consider the action of the BK move $t_2$ on the tableau
\[ T = \;
\begin{ytableau} 
    1 & 1 & 1 & 1 & 2 & 2 & 2 & 2 & 3 \\ 
    2 & 2 & 3 & 3 & 3 & 4 \\
    3 & 4 & 4 & 5
    \end{ytableau} \; .
\]
The skew tableau $S$ containing only boxes labelled $2$ and $3$ in $T$ is
\[ S = \;
\begin{ytableau} 
    \none & \none & \none & \none & 2 & *(yellow) 2 & *(yellow) 2 & *(yellow) 2 & *(pink) 3 \\
    2 & *(yellow) 2 & *(pink) 3 & *(pink) 3 & 3 \\
    3
    \end{ytableau} \; .
\]
Following Definition~\ref{defn:BK_CST}, observe that in the first row $b = 3$ (yellow boxes) and $c = 1$ (pink boxes), so after swapping we get $b' = 1$ and $c' = 3$. The second row has $b = 1$ and $c = 2$, so we turn a $3$ into a $2$, whereas the third row has $b = c = 0$ and hence is fixed. The new skew tableau $S'$ obtained after this process has the form
\[ S' = \;
\begin{ytableau} 
    \none & \none & \none & \none & 2 & *(pink) 2 & *(yellow) 3 & *(yellow) 3 & *(yellow) 3 \\
    2 & *(pink) 2 & *(pink) 2 & *(yellow) 3 & 3 \\
    3
    \end{ytableau} \; ,
\]
and hence
\[ t_2(T) = \;
\begin{ytableau} 
    1 & 1 & 1 & 1 & 2 & 2 & 3 & 3 & 3 \\ 
    2 & 2 & 2 & 3 & 3 & 4 \\
    3 & 4 & 4 & 5
    \end{ytableau} \; .
\]
\end{example}

Let $\mathcal{BK}$, called the {\it (combinatorial) Berenstein--Kirillov (BK) group}, denote the free group generated by the variables $t_i$ $\left(i \in \mathbb{Z}_{> 0} \right)$ modulo the relations satisfied by the BK moves $t_i$ when acting on all column-strict tableaux of all possible shapes. Note that there exist other variations of the BK moves, e.g., piecewise-linear BK moves on Gelfand--Tsetlin patterns \cite{berenstein-kirilov} and birational BK moves, which lead to different versions of the BK group. It is believed, but not proven, that they coincide. In this paper, we restrict our attention to the combinatorial version of these objects.

The relations in $\mathcal{BK}$ were first studied in a foundational paper of Berenstein and Kirillov \cite{berenstein-kirilov}, which included the following few:
\begin{enumerate}
    \item $t_i^2 = 1$ for all $i \ge 1$;
    \item $(t_i t_j)^2 = 1$ for all $i,j \ge 1$ with $|i - j| \ge 2$;
    \item $(t_1 q_i)^4 = 1$ for all $i\ge 3$, where $q_i = t_1 (t_2 t_1) \ldots (t_i t_{i-1} \ldots t_1)$; and
    \item $(t_1 t_2)^6 = 1$.
\end{enumerate}
More recently, Chmutov, Glick, and Pylyavskyy \cite{CGP} found new relations in this group generalizing Relations (3):
\begin{enumerate}\setcounter{enumi}{4}
    \item $(t_i q_{jk})^2 = 1$ whenever $i+1<j<k$, where $q_{jk} = q_{k-1} q_{k-j} q_{k-1}$,
\end{enumerate}
which interestingly give a group homomorphism from the cactus group $\mathcal{C}_n$ (cf. \cite{devadoss, henriques-kamnitzer,halacheva}) to the subgroup $\mathcal{BK}_n \subset \mathcal{BK}$ generated by $\{ t_1, \ldots, t_{n-1} \}$. For this reason, Relations (5) are called the {\it cactus relations}.
It was remarked by \cite{CGP} that Relation (4) is the only known relation in $\mathcal{BK}$ that does not follow from the relations in cactus groups.
Similar observations were made by Berenstein and Kirillov \cite{berenstein-kirillov-RIMS} using a purely group-theoretic approach. Note that $q_i$ consists of an iterated product of $\partial_j = t_j \dots t_2 t_1$; both of these have been previously studied as operators on column-strict tableaux, called {\it evacuation} and {\it promotion}\footnote{More precisely, the operators $\partial_i$ and $q_i$ as defined here are {\it partial} promotion and evacuation, as they act only on the subtableaux containing the entries $1, 2, \dots, i+1$, instead of the full tableaux.}
respectively, notably by \cite{SCHUTZENBERGER1972,schutzenberger1976,schutzenberger77,berenstein-kirilov}. It is known that evacuation $q_i$ is an involution, and so is $q_{jk}$. Thus the cactus relations are equivalent to the commutativity of $t_i$ and $q_{jk}$ for all $i+1 < j < k$.


\subsection{Bender--Knuth moves on linear extensions}\label{subsec:BK_LE}

Here, we define an analog of the Berenstein--Kirillov group acting on {\it linear extensions of posets}. This involves two steps. First, we specialize the action of the BK moves on column-strict tableaux to standard Young tableaux. Then, we view standard Young tableaux as linear extensions of a certain kind of poset. This allows us to generalize the BK moves on standard Young tableaux to act on linear extensions of arbitrary posets. Thus, our analog of the Berenstein--Kirillov group on linear extensions of posets is in one sense more specific and in another sense more general than the classical Berenstein--Kirillov group on column-strict tableaux. Note that the first step may (and indeed does) introduce new relations satisfied by the BK moves, whereas the second step may result in fewer relations. We discuss these claims in more detail below.

Let $(P, \le_P)$ be an $n$-element poset.

\begin{definition}\label{defn:linext}
A {\it linear extension} of $P$ is a linear order $(\ell,\le_\ell)$ which extends $P$ in the sense that $x \le_P y$ implies $x \le_\ell y$.
\end{definition}

More precisely, a linear extension $\ell$ of $P$ is a bijection $\ell : P \to \{1, \dots, n\} =: [n]$ such that $x <_P y$ implies $\ell(x) < \ell(y)$. Alternatively, a linear extension $\ell$ of $P$ can be interpreted as an ordered list $\ell = (p_1, \ldots, p_n)$ of elements of $P$ such that $p_i < p_j$ implies $i < j$; the index of an element $p$ on this list is precisely its label $\ell(p)$. The set of all linear extensions of $P$ is denoted by $\linext(P)$. The example below shows the poset $P$ defined by the relations $a,b < c,d$ (left) and a linear extension $\ell$ of $P$ (right). Observe that the labelling system on $\ell$ satisfies Definition~\ref{defn:linext}. In the alternative notation, we have $\ell = (b, a, d, c)$.

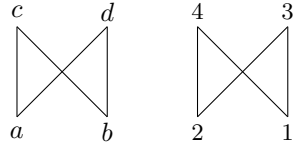
\begin{center}
    \begin{tikzpicture}
        \draw[black] (0,0) -- (0,1.2);
        \draw[black] (1.2,0) -- (1.2,1.2);
        \draw[black] (0,0) -- (1.2,1.2);
        \draw[black] (1.2,0) -- (0,1.2);
        \draw (0,-0.2) node[anchor=center] { $a$};
        \draw (1.2,-0.2) node[anchor=center] { $b$};
        \draw (0,1.4) node[anchor=center] { $c$};
        \draw (1.2,1.4) node[anchor=center] { $d$};
        \draw[black] (2.4,0) -- (2.4,1.2);
        \draw[black] (3.6,0) -- (3.6,1.2);
        \draw[black] (2.4,0) -- (3.6,1.2);
        \draw[black] (3.6,0) -- (2.4,1.2);
        \draw (2.4,-0.2) node[anchor=center] {\small $2$};
        \draw (3.6,-0.2) node[anchor=center] {\small $1$};
        \draw (2.4,1.4) node[anchor=center] {\small $4$};
        \draw (3.6,1.4) node[anchor=center] {\small $3$};
    \end{tikzpicture}
\end{center}

Given a partition $\lambda = (\lambda_1, \dots, \lambda_n)$, the {\it Young diagram} of $\lambda$ is a subset $D_{\lambda}$ of $\{1,2,\ldots\} \times \{1,2,\ldots\}$ defined by 
$$
D_{\lambda} = \{(i, j) : 1 \leq j \leq \lambda_i\}.
$$
The \textit{Ferrers poset} $F_\lambda$ of $\lambda$ is defined to be the set $D_\lambda$ with a partial order generated by the covering relations $(i,j) < (i, j+1)$ and $(i,j) < (i+1, j)$. It is well-known that there is a one-to-one correspondence between the standard Young tableaux of shape $\lambda$ and the linear extensions of the corresponding Ferrers poset $F_\lambda$.
The example below shows a standard Young tableau of shape $\lambda = (4,2,1)$ and its corresponding linear extension of the Ferrers poset $F_{\lambda}$.
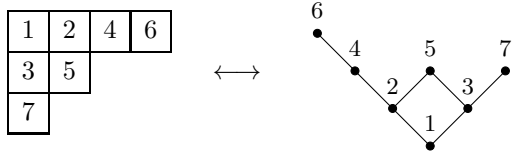
\[ \;
\begin{ytableau} 
    1 & 2 & 4 & 6 \\ 
    3 & 5 \\
    7 
    \end{ytableau} \;
\quad \longleftrightarrow \quad
\;
\vcenter{\hbox{\begin{tikzpicture}
    \centering
        \filldraw[black] (0,-2) circle (1.5pt);
        \filldraw[black] (-0.5,-1.5) circle (1.5pt);
        \filldraw[black] (0.5,-1.5) circle (1.5pt);
        \filldraw[black] (-1,-1) circle (1.5pt);
        \filldraw[black] (0,-1) circle (1.5pt);
        \filldraw[black] (1,-1) circle (1.5pt);
        \filldraw[black] (-1.5,-0.5) circle (1.5pt);
        
        \draw[black] (0,-2) -- (-1.5,-0.5);
        \draw[black] (0,-2) -- (1,-1);
        \draw[black] (-0.5,-1.5) -- (0,-1);
        \draw[black] (0.5,-1.5) -- (0,-1);

        \draw (0,-1.7) node[anchor=center] {\small$1$};
        \draw (-0.5,-1.2) node[anchor=center] {\small$2$};
        \draw (0.5,-1.2) node[anchor=center] {\small$3$};
        \draw (-1,-0.7) node[anchor=center] {\small$4$};
        \draw (0,-0.7) node[anchor=center] {\small$5$};
        \draw (1,-0.7) node[anchor=center] {\small$7$};
        \draw (-1.5,-0.2) node[anchor=center] {\small$6$};

    \end{tikzpicture}}}
 \;
\]

\smallskip

Observe that if $T$ is a standard Young tableau, by Definition~\ref{defn:BK_CST}, the BK move $t_i$ acts on $T$ by simply switching $i$ and $i+1$ if they label non-adjacent boxes and fixing $T$ otherwise. Via the above bijection, there is an analog of the BK moves on linear extensions of Ferrers posets, which naturally extends to any arbitrary poset. This generalization was first introduced by Haiman \cite{haiman} as well as Malvenuto and Reutenauer \cite{malvenuto-reutenauer} to study promotion and evacuation on linear extensions, which are analogs of operators on column-strict tableaux previously studied by Sch\"utzenberger~\cite{SCHUTZENBERGER1972, schutzenberger1976, schutzenberger77}.

\begin{definition}
For a given $n$-element poset $P$ and $1 \le i \le n-1$, the {\it Bender--Knuth move} $t_i$ is an involution of $\linext(P)$ that sends $\ell = (p_1, \ldots, p_n)$ to
\[(p_1, \ldots, p_{i+1}, p_i, \ldots, p_n)\]
if $p_i$ and $p_{i+1}$ are incomparable, and fixes $\ell$ otherwise.
\end{definition}

As in the case with column-strict tableaux, two important operators generated by the BK moves are (partial) promotion $\partial_i = t_{i}t_{i-1}\ldots t_1$ (by convention, $\partial_0 = 1$) and (partial) evacuation $q_i = \partial_0\partial_1\dots\partial_i$.
We refer the reader to \cite{Stanley-promotion-evacuation} for a careful analysis of these operators; note that the inclusion of the identity operator $\partial_0$ in $q_i$ is motivated by a procedural description in the reference of the evacuation $q_i$ as a series of $i+1$ partial promotions.
For the purpose of this paper, we emphasize a useful procedural description of the promotion operator $\partial_i$:

\begin{definition}\label{defn:promo_Stanley}
Given an $n$-element poset $P$ and $1 \le i \le n-1$, the promotion operator $\partial_i$ is a permutation of $\linext(P)$ that sends $\ell \in \linext(P)$ to the linear extension of $P$ obtained from the following procedure:
\begin{enumerate}
    \item Consider the subposet formed by the elements $\ell^{-1}(1),\ell^{-1}(2),\ldots,\ell^{-1}(i+1)$;
    \item Let $p_1 = \ell^{-1}(1)$. Remove the label $1$ from $p_1$;
    \item Among the elements covering $p_1$, let $p_2$ be the element with the smallest label $\ell(p_2)$;
    \item Slide this label down to $p_1$;
    \item Repeat until we reach a maximal element $p_k$. Label $p_k$ with $i+2$ and subtract all labels on the entire subposet by $1$.
\end{enumerate}
\end{definition}

We call $p_1<p_2<\ldots<p_k$ a \textit{promotion chain}. An immediate consequence of this interpretation is that when a promotion operator $\partial_i$ acts on a linear extension of a disjoint union of posets $P_1,\ldots,P_n$, the only component $P_i$ whose labels' relative order may be affected is the one containing $\ell^{-1}(1)$.

Given a finite poset $P$, its {\it linear extension graph} is the graph whose vertices are labelled by linear extensions of $P$ and edges are given by the BK moves that swap corresponding linear extensions (e.g., see Figure~\ref{fig:BK_moves_posets}). Linear extension graphs were first introduced by Pruesse and Ruskey \cite{pr1991}, and previously used in the study of linear extension generation (e.g., \cite{RUSKEY1992,stachowiak1992,west1993,naatz2000,bm2013}) and Markov chains on $\linext(P)$ (e.g., \cite{ayyer-klee-schilling}).

\begin{figure}[htp]
    \centering
    \begin{tikzpicture}
        \draw[black] (-2.4,1.2) -- (-2.4,2.4);
        \draw[black] (-1.2,1.2) -- (-1.2,2.4);
        \draw[black] (-2.4,1.2) -- (-1.2,2.4);
        \draw[black] (-1.2,1.2) -- (-2.4,2.4);
        \draw (-2.4,1) node[anchor=center] {\small $a$};
        \draw (-1.2,1) node[anchor=center] {\small $b$};
        \draw (-2.4,2.6) node[anchor=center] {\small $c$};
        \draw (-1.2,2.6) node[anchor=center] {\small $d$};
        \draw[black] (0,0) -- (0,1.2);
        \draw[black] (1.2,0) -- (1.2,1.2);
        \draw[black] (0,0) -- (1.2,1.2);
        \draw[black] (1.2,0) -- (0,1.2);
        \draw (0,-0.2) node[anchor=center] {\small $2$};
        \draw (1.2,-0.2) node[anchor=center] {\small $1$};
        \draw (0,1.4) node[anchor=center] {\small $3$};
        \draw (1.2,1.4) node[anchor=center] {\small $4$};
        \draw[black] (0.6,1.4) -- (0.6,2.2);
        \draw (0.85,1.8) node[anchor=center] {\small$t_1$};
        \draw[black] (0,2.4) -- (0,3.6);
        \draw[black] (1.2,2.4) -- (1.2,3.6);
        \draw[black] (0,2.4) -- (1.2,3.6);
        \draw[black] (1.2,2.4) -- (0,3.6);
        \draw (0,2.2) node[anchor=center] {\small $1$};
        \draw (1.2,2.2) node[anchor=center] {\small $2$};
        \draw (0,3.8) node[anchor=center] {\small $3$};
        \draw (1.2,3.8) node[anchor=center] {\small $4$};
        \draw[black] (2.4,0) -- (2.4,1.2);
        \draw[black] (3.6,0) -- (3.6,1.2);
        \draw[black] (2.4,0) -- (3.6,1.2);
        \draw[black] (3.6,0) -- (2.4,1.2);
        \draw (2.4,-0.2) node[anchor=center] {\small $2$};
        \draw (3.6,-0.2) node[anchor=center] {\small $1$};
        \draw (2.4,1.4) node[anchor=center] {\small $4$};
        \draw (3.6,1.4) node[anchor=center] {\small $3$};
        \draw[black] (3,1.4) -- (3,2.2);
        \draw (3.25,1.8) node[anchor=center] {\small$t_1$};
        \draw[black] (2.4,2.4) -- (2.4,3.6);
        \draw[black] (3.6,2.4) -- (3.6,3.6);
        \draw[black] (2.4,2.4) -- (3.6,3.6);
        \draw[black] (3.6,2.4) -- (2.4,3.6);
        \draw (2.4,2.2) node[anchor=center] {\small $1$};
        \draw (3.6,2.2) node[anchor=center] {\small $2$};
        \draw (2.4,3.8) node[anchor=center] {\small $4$};
        \draw (3.6,3.8) node[anchor=center] {\small $3$};
        \draw[black] (1.4,0.6) -- (2.2,0.6);
        \draw (1.8,0.85) node[anchor=center] {\small $t_3$};
        \draw[black] (1.4,3) -- (2.2,3);
        \draw (1.8,3.25) node[anchor=center] {\small$t_3$};
    \end{tikzpicture}
    \caption{Linear extension graph of the poset defined by the relations $a,b < c,d$.}
    \label{fig:BK_moves_posets}
\end{figure}
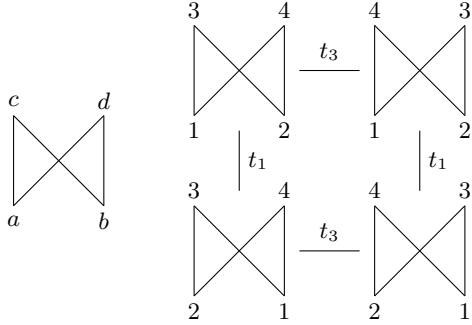

We may define an analog of the BK group in this context.

\begin{definition}\label{defn:BKP}
For a given $n$-element poset $P$, the {\it Berenstein--Kirillov group of }$P$, denoted by $\bkp$, is the permutation group of the set $\linext(P)$ generated by the BK moves $t_i$ ($1 \le i \le n-1$).
\end{definition}

Very few studies have investigated this group explicitly, e.g., \cite{vershik-tsilevich,vershik21}. Some relations are known to hold in $\bkp$ for all posets $P$, including
\begin{enumerate}
    \item $t_i^2 = 1$ for all $i \ge 1$;
    \item $(t_i t_j)^2 = 1$ for all $i,j \ge 1$ with $|i - j| \ge 2$; and
    \item $(t_i t_{i+1})^6 = 1$ for all $i \ge 1$ \cite{Stanley-promotion-evacuation}.
\end{enumerate}
When $P$ is a Ferrers poset, the inclusion of Relations (3) represents new relations satisfied by the generators of the classical BK group when specialized to acting on standard Young tableaux. Furthermore, the result of \cite{CGP} implies that the cactus relations hold in $\bkp$ in this case, alongside the above. Outside of these few exceptions, very little is known about the group relations in $\bkp$.

Note that there is no direct correspondence between the relations in the groups $\mathcal{BK}$ and $\bkp$ for arbitrary posets $P$. This is because our definition of $\bkp$ involves first specializing the action of the BK moves to standard Young tableaux (which may create more relations among the generators $t_i$, e.g., Relations (3) above) then generalizing to an action on linear extensions of arbitrary posets (which may eliminate some relations). This observation leads to the first motivating question of our paper:

\begin{question}\label{ques:ker}
What further relations hold in $\bkp$, for which posets $P$? In particular, for which posets $P$ do the cactus relations hold in $\bkp$?
\end{question}

This question will be addressed in Section~\ref{sec:rel}. In general, one of the main takeaways from our investigation is that not all relations in $\mathcal{BK}$ continue to hold in $\bkp$ for arbitrary posets $P$.

Similarly to its relations, very little has been established about the properties of $\bkp$ as a subgroup of the symmetric group on $\linext(P)$. The following fact about $\bkp$ as a permutation group of $\linext(P)$ has been widely utilized in different forms by, e.g., \cite{pak01,lam-williams,develin-macauley-reiner,defant-kravitz}. A proof was explicitly given by Ayyer, Klee, and Schilling (see Proposition 4.1 of \cite{ayyer-klee-schilling}), by showing that linear extension graphs are strongly connected.

\begin{prop}\label{prop:BKP_transitive}
    $\bkp$ is a transitive subgroup of $\symm_{\linext(P)}$.
\end{prop}

Another rare instance is Vershik and Tsilevich's study of the BK groups of Ferrers posets in the form of permutation groups on Young graphs \cite{vershik-tsilevich}. Motivated by this lack of investigation, the second part of our paper will be guided by the following question about $\bkp$:

\begin{question}\label{ques:im}
For which posets $P$ does $\bkp = \symm_{\linext(P)}$? More generally, what can we say about the relative size of $\bkp$ as a transitive subgroup of $\symm_{\linext(P)}$?
\end{question}

This question will be studied in Sections~\ref{sec:symm} and ~\ref{sec:stab}.


\subsection{Poset operations}\label{poset-operations}\label{subsec:poset_op}
We fix the notations for various poset operations discussed in this paper.
Let $P$ and $Q$ be finite posets.
Let $P \oplus Q$ denote the ordinal sum of $P$ and $Q$, where each element of $P$ is less than every element of $Q$. Note that the ordinal sum operation is not commutative; in general, $P \oplus Q \neq Q \oplus P$.
Let $P + Q$ denote the disjoint union of $P$ and $Q$. On the contrary, this operation is commutative.
Finally, let $P^*$ be the dual of the poset $P$, defined by inverting the order of $P$ in the sense that $x \le_{P^*} y$ if and only if $y \le_P x$.


\section{Relations in \texorpdfstring{$\bkp$}{BK\_P}}\label{sec:rel}

In this section, we discuss the group relations in $\bkp$, including the trivialization, braid, and cactus relations. The most important findings are located in Section~\ref{subsec:cactus}, where we study posets for which the cactus relations hold in their BK groups and study their behaviors under disjoint union and ordinal sum operations, and Section~\ref{subsec:cactus_jdt}, where we study several potential families of such posets.


\subsection{Relations in \texorpdfstring{$\bkp$}{BK\_P} and convex induced subposets}\label{subsec:rel_BKP_cvx}

Let $P$ be an $n$-element poset. Recall that a {\it relation} in $\bkp$ is an equation $w = 1$, where $w$ is a word in the alphabet $\{t_1, \dots, t_{n-1}\}$. The first natural family of relations in $\bkp$ that we will discuss is the relations $t_i = 1$, called the {\it trivialization of} $t_i$.

\begin{prop} \label{prop:rel_ti=1}
    For $1 \le i \le n-1$, the relation $t_i = 1$ holds in $\bkp$ if and only if $P$ can be written as an ordinal sum $P = P_1 \oplus P_2$ where $|P_1|=i$.
\end{prop}

\begin{proof}
When $P=P_1 \oplus P_2$ with $|P_1|=i$, every linear extension $\ell$ of $P$ has $\ell(P_1)=\{1,2,\ldots,i\}$ and $\ell(P_2)=\{i+1,i+2,\ldots,n\}$, so 
$\ell^{-1}(i) <_P \ell^{-1}(i+1)$, and $t_i$ fixes $\ell$.

Conversely, if $P \neq P_1 \oplus P_2$ whenever $|P_1|=i$, we claim that for any order ideal (lower set) $I$ of $P$ with cardinality $|I| = i$, there exist a maximal element $u$ of $I$ and a minimal element $v$ of the upper set $P \setminus I$ such that $u$ and $v$ are incomparable.
Suppose not, then for every maximal element $u$ of $I$ and every minimal element $v$ of $P \setminus I$, we have $u <_P v$, as $u >_P v \notin I$ is not allowed given that $I$ is an order ideal. Therefore, $P = I \oplus (P \setminus I)$, a contradiction.

Let $I$ be an order ideal of $P$ with cardinality $|I|=i$; for example, choose the inverse image $I=f^{-1}(\{1,2,\ldots,i\})$ for any $f \in \linext(P)$. We now construct a linear extension $\ell$ of $P$ such that $t_i(\ell) \neq \ell$. By the previous claim, we may find an incomparable pair of elements $\{u,v\}$ of $P$ where $u$ is maximal in $I$ and $v$ is minimal in $P \setminus I$. Construct a linear extension $\ell$ by first labeling $\ell(u)=i$ and $\ell(v)=i+1$, then using arbitrary linear extensions to label $I \setminus \{u\}$ with $\{1,2,\ldots, i-1\}$ 
and $(P \setminus I) \setminus \{v\}$ with $\{i+2,i+3,\ldots, n\}$. It is easy to check that $t_i$ swaps the labels of $u$ and $v$ in $\ell$, and hence $t_i (\ell) \ne \ell$.
\end{proof}

\begin{definition}
    A poset is \emph{(ordinally) indecomposable} if it is not an ordinal sum of two or more non-empty posets.
\end{definition}

\begin{cor}\label{ti-indep}
    Let $|P| = n$. If $P$ is indecomposable, then $\{t_1, \dots, t_{n-1}\}$ is an inclusion-minimal generating set in $\bkp$.
\end{cor}

\begin{proof}
    By Proposition \ref{prop:rel_ti=1}, there is some linear extension $\ell$ on which $t_i$ is not in the stabilizer. Let $|P| = n$. Suppose for contradiction that $w\ell = t_i \ell$ for some $w \in \ang{t_1, t_2, \dots, \hat{t_i}, \dots, t_{n-1}}$. Then consider the order ideal $I$ such that $\ell(I) = [i]$, i.e., the elements of $P$ with labels $\{1, 2, \dots, i\}$. Then any action that is not $t_i$ preserves this image $\ell(I)$, so $i+1 \notin w\ell(I)$. However, we have $i+1 \in t_i\ell (I)$, a contradiction.
\end{proof}

Next, we present a useful tool for studying the relations in $\bkp$ by examining the convex induced subposets of $P$. Recall that an {\it induced subposet} $Q \subseteq P$ is a subset of vertices in $P$ such that for any $x,y \in Q$, $x \le_Q y$ if and only if $x \le_P y$. A {\it convex induced subposet} is an induced subposet such that if $x,z \in Q$ and $y \in P$ satisfy $x \le_P y \le_P z$, then $y \in Q$.

\begin{definition}
    A \emph{relation type} in $\bkp$ is a set of relations $w_i = 1$, where $w_i$ is obtained from a fixed word $w$ in $\{t_1, \dots, t_{n-1}\}$ by translating the indices of all generators $t_j$ in $w$ by an integer $i$, such that $1 \le i+j \le n-1$ for all $j$.
\end{definition}

For example, $(t_1 t_2)^6=1$ is a relation, whereas the relations $(t_i t_{i+1})^6=1$ form a relation type. The following proposition gives a means to generate relation types in the BK groups for convex induced subposets.

\begin{prop} \label{prop:rel_cvx}
    If a relation type holds in $\bkp$, then it also holds in $\bkq$ for every convex induced subposet $Q$ of $P$.
\end{prop}

The proof of this statement makes use of the following lemma.

\begin{lemma} \label{lem:cvx_ind}
    Let $Q$ be a convex induced subposet of $P$ and $\ell_1 \in \linext(Q)$. Then there is a linear extension $\ell_2 \in \linext(P)$ such that for some $i \in \mathbb{Z}_{\ge 0}$, $\ell_2(v) = \ell_1(v) + i$ for all $v \in Q$.
\end{lemma}

\begin{proof}
Disjointly decompose the poset $P$ into three subposets $I$, $Q$, and $F$, where 
    $$I:=\{p \in P \setminus Q: \text{there exists } q \in Q \text{ with }q > p\}, \text{ and }    F:= P \setminus (Q \sqcup I).$$
Then the subposet $Q \sqcup I$ of $P$ containing the
elements of $Q$ and $I$ will be an order ideal of $P$, and because $Q$ is convex, $I$ is an order ideal of $Q \sqcup I$.  Letting $i:=|I|, q:=|Q|,$ and $f:=|F|$,  one can therefore form a linear extension $\ell_2$ of $P$ having $\ell_2^{-1}(\{1,2,\ldots,i\})$
agree with any linear extension of $I$, having 
$\ell_2^{-1}(\{i+1,i+2,\ldots,i+q\})$
agree with $\ell_1$ after adding $i$ to the values in $\ell_1$, and having
$\ell_2^{-1}(\{i+q+1,i+q+2,\ldots,i+q+f\})$
agree with any linear extension of $F$
after adding $i+q$ to its values.
\end{proof}

\begin{proof}[Proof of Proposition~\ref{prop:rel_cvx}]
    We prove the contrapositive. Suppose there is a convex induced subposet $Q \subseteq P$ for which the relation type does not hold in $\bkq$. That is, there exists a relation $w(t_i, \ldots, t_{i+k}) = 1$ of this type which fails on a linear extension $\ell_1$ of $Q$. By Lemma~\ref{lem:cvx_ind}, there is a linear extension $\ell_2 \in \linext(P)$ such that for some $j \in \mathbb{Z}_{\ge 0}$, $\ell_2(v) = \ell_1(v) + j$ for all $v \in Q$. It is now easy to check that $w(t_{i+j},\ldots,t_{i+j+k}) = 1$ fails on $\ell_2$.
\end{proof}

Lemma~\ref{lem:cvx_ind} directly implies an interesting fact about convex induced subposets:

\begin{cor}\label{cor:BKP_cvx_ind}
    If $Q$ is a convex induced subposet of $P$, then there is an injection $\bkq \hookrightarrow \bkp$.
\end{cor}

As an application of Proposition~\ref{prop:rel_cvx}, we examine a well-known relation type called the {\it braid relations} in $\bkp$ for an $n$-element poset $P$. Recall that the braid relations are of the form $(t_i t_{i+1})^3 = 1$ for all $i=1,\ldots,n-2$. The following fact was previously observed by Ayyer, Klee, and Schilling (see Proposition 2.2 of \cite{ayyer-klee-schilling}), but with the proof omitted.

\begin{prop}\label{prop:rel_braid}
    The braid relations hold in $\bkp$ if and only if $P$ is a disjoint union of chains.
\end{prop}

\begin{proof}
    One can check that the only posets of cardinality $3$ that fail to satisfy the relation $(t_1 t_2)^3=1$ are defined by either $a > b < c$ (V-shaped) or $a < b > c$ (inverted V-shaped). Any poset that contains one of these two posets as an induced subposet contains it as a convex induced subposet, and hence will fail to satisfy $(t_i t_{i+1})^3=1$ for some $i$ by Proposition~\ref{prop:rel_cvx}. The posets that do not contain one of these two posets are precisely disjoint unions of chains. Conversely, suppose, by contradiction, that a disjoint union of chains fails to satisfy $(t_i t_{i+1})^3=1$ for some $i$. Then the induced subposet on the elements with labels $i$, $i+1$, and $i+2$ fails to satisfy $(t_1 t_2)^3=1$, and hence is one of the two posets mentioned above. But a disjoint union of chains does not have such an induced subposet, so we have a contradiction.
\end{proof}

Since the braid relations are the defining relations of the symmetric group $\symm_n$ and the braid group $B_n$, it follows that the natural map (of sets) that sends the generator $\sigma_i$ of $\symm_n$ (resp. $B_n$) to $t_i$ for all $1 \le i \le n-1$ gives a well-defined action of $\symm_n$ (resp. $B_n$) on $\linext(P)$ only when $P$ is a disjoint union of chains.


\subsection{Cactus relations in $\mathcal{BK}_P$}\label{subsec:cactus}

Another important family of relations involving the generators $t_i$ is the cactus relations, which have the form $(t_i q_{jk})^2 = 1$, where $q_{jk} = q_{k-1}q_{k-j}q_{k-1}$ and $q_j = t_1(t_2t_1)\ldots(t_{j} t_{j-1} \ldots t_1)$ for any $i+1 < j < k$. Recall that when $P$ is a Ferrers poset, the cactus relations hold in the group $\bkp$ for all $i+1 < j < k \le |P|$. One reasonable question is whether the same statement is true for an arbitrary poset. Unfortunately, that is not the case.

\begin{example} \label{min-non-cactus-ex}
On the $4$-element posets, there is only one eligible cactus relation: $(t_1 q_{34})^2=1$.  The only $4$-element connected posets for which this relation fails are the followings:

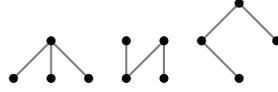
\begin{center}
\begin{tikzpicture}
\draw[gray, thick] (0,0) -- (0.5,0.5);
\draw[gray, thick] (0.5,0) -- (0.5,0.5);
\draw[gray, thick] (1,0) -- (0.5,0.5);
\draw[gray, thick] (1.5,0) -- (1.5,0.5);
\draw[gray, thick] (1.5,0) -- (2,0.5);
\draw[gray, thick] (2,0) -- (2,0.5);
\draw[gray, thick] (3,0) -- (2.5,0.5);
\draw[gray, thick] (2.5,0.5) -- (3,1);
\draw[gray, thick] (3,1) -- (3.5,0.5);
\filldraw[black] (0,0) circle (1.5pt);
\filldraw[black] (0.5,0) circle (1.5pt);
\filldraw[black] (1,0) circle (1.5pt);
\filldraw[black] (0.5,0.5) circle (1.5pt);
\filldraw[black] (1.5,0) circle (1.5pt);
\filldraw[black] (2,0) circle (1.5pt);
\filldraw[black] (3,0) circle (1.5pt);
\filldraw[black] (1.5,0.5) circle (1.5pt);
\filldraw[black] (2,0.5) circle (1.5pt);
\filldraw[black] (2.5,0.5) circle (1.5pt);
\filldraw[black] (3,1) circle (1.5pt);
\filldraw[black] (3.5,0.5) circle (1.5pt);
\end{tikzpicture}
\end{center}    

\noindent For a counterexample with $5$ elements, the following poset exhibits an interesting property that all eligible cactus relations fail in its BK group:

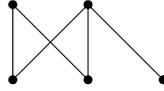
\vspace{.2cm}
\begin{center}
\begin{tikzpicture}
        
        \filldraw[black] (5,0) circle (1.5pt);
        \filldraw[black] (4,0) circle (1.5pt);
        \filldraw[black] (5,1) circle (1.5pt);
        \filldraw[black] (4,1) circle (1.5pt);
        \filldraw[black] (6,0) circle (1.5pt);
        
        \draw[black] (5,0) -- (4,1);
        \draw[black] (4,0) -- (4,1);
        \draw[black] (5,0) -- (5,1);
        \draw[black] (4,0) -- (5,1);
        \draw[black] (6,0) -- (5,1);
        
        
    \end{tikzpicture}
\end{center}
\end{example}

The fact that cactus relations do not hold in general on all linear extensions of all posets presents a major deviation from the original BK group $\mathcal{BK}$ on semi-standard Young tableaux, a subset of which--namely the standard Young tableaux--corresponds to linear extensions of Ferrers posets.
It remains of interest to examine when the cactus relations hold in $\bkp$, i.e., when there is a natural action of the cactus group $\mathcal{C}_{|P|}$ on $\linext(P)$. We have already encountered one family of posets for which this holds true.

\begin{thm}[\cite{CGP}]
    The cactus relations hold in $\bkp$ for all Ferrers posets $P$.
\end{thm}

For the rest of this section, we will study properties of posets $P$ for which the cactus relations hold in $\bkp$. We will also identify several sufficient conditions for such posets. As a preview, later in this section, we will show that starting with an arbitrary poset $P$, we can force the cactus relations to hold or to not hold in $\bkp$ by taking an ordinal sum with a sufficiently large chain (Proposition~\ref{prop:oplus_chain}) or with a sufficiently large antichain (Proposition~\ref{prop:ord-sum-not-cactus}).

First, observe that Proposition~\ref{prop:rel_cvx} does not apply to the cactus relations, since they do not form a relation type: implicitly $q_{jk}$ contains the operators $q_j = t_1(t_2t_1)\ldots(t_{j} t_{j-1} \ldots t_1)$ which always involve $t_1$. 
However, there is an analogous statement for when the cactus relations hold in $\bkp$.

\begin{prop} \label{prop:cactus_ideal}
    The cactus relations hold in $\bkp$ if and only if they hold in $\mathcal{BK}_I$ for every order ideal $I$ of $P$.
\end{prop}

\begin{proof}
    By contrapositive, suppose that there exists an order ideal $I$ of $P$ such that not all applicable cactus relations hold in $\mathcal{BK}_I$. Let $|I| = m$, then there exists a triple $(i, j, k)$ where $2 \leq i+1 < j < k \leq m$ and a linear extension $\ell_1$ of $I$, such that $(t_i q_{jk})^2 (\ell_1) \ne \ell_1$. From $\ell_1$, we can construct a linear extension $\ell$ of $P$ by setting $\ell = \ell_1$ on $I$ and $\ell = \ell_2 + m$ on $P \setminus I$ for some linear extension $\ell_2$ of the induced subposet $P \setminus I$. Thus, by the above construction, $(t_i q_{jk})^2 (\ell) \ne \ell$, so the cactus relation $(t_i q_{jk})^2 =1$ does not hold in $\bkp$. The converse immediately follows from the fact that $P$ is an order ideal of itself.
\end{proof}

This result is especially useful in eliminating posets $P$ for which some cactus relations do not hold in $\bkp$ by identifying (small) ideals sitting in them for which the same holds.


\subsubsection{Disjoint unions}\label{ssubsec:cactus_disj}

Let $P$ and $Q$ be finite posets.
We will show that if the cactus relations hold in $\mathcal{BK}_P$ and $\mathcal{BK}_Q$, then they hold in $\mathcal{BK}_{P+Q}$.

We first define the following map on linear extensions of the disjoint union $P+Q$ to help break down showing commutativity of $t_i$ and $q_{jk}$ on an entire linear extension of $P + Q$.

\begin{definition}\label{defn:part_disj_union}
    Let $P$ and $Q$ be posets with $|P| = m$ and $|Q| = n$.
    Define the map
    \begin{align*}
        T : \linext(P+Q) &\to \linext(P) \times \linext(Q) \times \binom{[m+n]}{m} \times \binom{[m+n]}{n} \\
        \text{by} \hspace{0.93in} \ell &\mapsto (\ell_P,\ell_Q,S(P),S(Q)),
    \end{align*}
    where $S(P)$ and $S(Q)$ are the sets of labels in $\ell$ on $P$ and $Q$, and $\ell_P$ and $\ell_Q$ are independent linear extensions on $P$ and $Q$ which agree with $\ell$, in the sense that for any $p_1,p_2 \in P$, $\ell_P(p_1) \le \ell_P(p_2)$ if and only if $\ell(p_1) \le \ell(p_2)$, and similarly for $Q$.

    It is not difficult to see that $T$ is an injection. It is also straightforward to transfer the action of $t_i$ on $\ell$ to an action on the tuple $(\ell_P, \ell_Q, S(P), S(Q))$. If $\ell^{-1}(i)$ is in $P$ and $\ell^{-1}(i+1)$ is in $Q$ or vice versa, then $t_i$ swaps the label $i$ in $S(P)$ with $i+1$ in $S(Q)$ while keeping $\ell_P$ and $\ell_Q$ the same. If $i$ and $i+1$ are both in $P$, then $t_i$ stabilizes $S(P), S(Q)$, and $\ell_Q$, while acting on $\ell_P$ by $t_{i'}$ where $i' = |S(P) \cap [i]|$.
\end{definition}

\begin{example}
    Consider the following linear extension $\ell$ of a disjoint union $P+Q$ (where the posets $P$ and $Q$ underlie the left and right components, respectively):
    
    
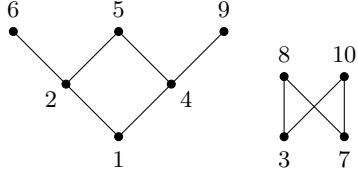
\centering
    \begin{tikzpicture}
        
        \filldraw[black] (0,0) circle (1.5pt);
        \filldraw[black] (-0.7,0.7) circle (1.5pt);
        \filldraw[black] (0.7,0.7) circle (1.5pt);
        \filldraw[black] (-1.4,1.4) circle (1.5pt);
        \filldraw[black] (0,1.4) circle (1.5pt);
        \filldraw[black] (1.4,1.4) circle (1.5pt);
        
        \filldraw[black] (3,0) circle (1.5pt);
        \filldraw[black] (2.2,0) circle (1.5pt);
        \filldraw[black] (3,0.8) circle (1.5pt);
        \filldraw[black] (2.2,0.8) circle (1.5pt);
        
        \draw[black] (0,0) -- (1.4,1.4);
        \draw[black] (0,0) -- (-1.4,1.4);
        \draw[black] (0.7,0.7) -- (0,1.4);
        \draw[black] (-0.7,0.7) -- (0,1.4);
        
        \draw[black] (3,0) -- (2.2,0.8);
        \draw[black] (2.2,0) -- (2.2,0.8);
        \draw[black] (3,0) -- (3,0.8);
        \draw[black] (2.2,0) -- (3,0.8);
        
        \draw (0,-0.3) node[anchor=center] {\small$1$};
        \draw (0.9,0.5) node[anchor=center] {\small$4$};
        \draw (-0.9,0.5) node[anchor=center] {\small$2$};
        \draw (-1.4,1.7) node[anchor=center] {\small$6$};
        \draw (0,1.7) node[anchor=center] {\small$5$};
        \draw (1.4,1.7) node[anchor=center] {\small$9$};
        
        \draw (2.2,-0.3) node[anchor=center] {\small$3$};
        \draw (3,-0.3) node[anchor=center] {\small$7$};
        \draw (2.2,1.1) node[anchor=center] {\small$8$};
        \draw (3,1.1) node[anchor=center] {\small$10$};
    
    \end{tikzpicture}
    
    \justifying
    \noindent The map $T$ sends this linear extension to the tuple containing linear extensions $\ell_P$ and $\ell_Q$:

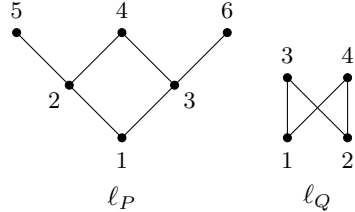
    \centering
    \begin{tikzpicture}
        
        \filldraw[black] (6,0) circle (1.5pt);
        \filldraw[black] (5.3,0.7) circle (1.5pt);
        \filldraw[black] (6.7,0.7) circle (1.5pt);
        \filldraw[black] (4.6,1.4) circle (1.5pt);
        \filldraw[black] (6,1.4) circle (1.5pt);
        \filldraw[black] (7.4,1.4) circle (1.5pt);
        
        \filldraw[black] (9,0) circle (1.5pt);
        \filldraw[black] (8.2,0) circle (1.5pt);
        \filldraw[black] (9,0.8) circle (1.5pt);
        \filldraw[black] (8.2,0.8) circle (1.5pt);
        
        \draw[black] (6,0) -- (7.4,1.4);
        \draw[black] (6,0) -- (4.6,1.4);
        \draw[black] (6.7,0.7) -- (6,1.4);
        \draw[black] (5.3,0.7) -- (6,1.4);
        
        \draw[black] (9,0) -- (8.2,0.8);
        \draw[black] (8.2,0) -- (8.2,0.8);
        \draw[black] (9,0) -- (9,0.8);
        \draw[black] (8.2,0) -- (9,0.8);
        
        \draw (6,-0.3) node[anchor=center] {\small$1$};
        \draw (6.9,0.5) node[anchor=center] {\small$3$};
        \draw (5.1,0.5) node[anchor=center] {\small$2$};
        \draw (4.6,1.7) node[anchor=center] {\small$5$};
        \draw (6,1.7) node[anchor=center] {\small$4$};
        \draw (7.4,1.7) node[anchor=center] {\small$6$};
        
        \draw (8.2,-0.3) node[anchor=center] {\small$1$};
        \draw (9,-0.3) node[anchor=center] {\small$2$};
        \draw (8.2,1.1) node[anchor=center] {\small$3$};
        \draw (9,1.1) node[anchor=center] {\small$4$};

        \draw (6,-0.8) node[anchor=center] {$\ell_P$};
        \draw (8.6,-0.8) node[anchor=center] {$\ell_Q$};
    
    \end{tikzpicture}
    
    \justifying
    \noindent and the sets of labels $S(P)=\{1,2,4,5,6,9\}$ and $S(Q)=\{3,7,8,10\}$.
    
    Observe that if we act on $\ell$ by $t_3$, the resulting $\ell_P$ and $\ell_Q$ stays the same, while we have new label sets $S'(P) = \{1,2,3,5,6,9\}$ and $S'(Q) = \{4,7,8,10\}$, with $3$ and $4$ swapped. On the other hand, if we act on $\ell$ by $t_5$, the resulting $S(P), S(Q)$, and $\ell_Q$ stay the same while $\ell_P'=t_4(\ell_P)$.
\end{example}

The following lemma describes the action of $q_{i-1}$ on the label sets $S(P)$ and $S(Q)$.

\begin{lemma} \label{qi-action-disjoint-union}
    Let $P,Q$ be finite posets and fix a linear extension $\ell \in \linext(P+Q)$. Let $1 \le i \le |P+Q|$, then
    \begin{enumerate}[i)]
        \item For all $j \le i$, $j \in S(P)$ if and only if $i-j+1 \in (q_{i-1}S)(P)$;
        \item For all $j > i$, $j \in S(P)$ if and only if $j \in (q_{i-1}S)(P)$.
    \end{enumerate}
\end{lemma}

\begin{proof}
    The converse of each statement follows directly from the fact that $q_{i-1}$ is an involution. Note that (ii) is immediate, as $q_{i-1}$ only affects the labels in $[i]$, so $j>i$ remains a label in $(q_{i-1}S)(P)$. It remains to show (i).
    
    Consider the action of $q_{i-1} = \partial_0\partial_1 \partial_2 \dots \partial_{i-1}$ on a fixed linear extension $\ell \in \linext(P+Q)$, where $\partial_j = t_{j}t_{j-1} \dots t_1$ is the promotion operator. Observe that by Definition~\ref{defn:promo_Stanley}, if $1 \in S(P)$, then we have $j+1 \in (\partial_{j} S)(P)$. On the other hand, if $k \in S(P)$ and $1 < k \le j$, then $k-1 \in (\partial_j S)(P)$.
    Let $j \in S(P)$ where $j \le i$. After applying the first $j-1$ promotion operators in $q_{i-1}$, by this observation, we have $1 \in (\partial_{i-j+1} \dots \partial_{i-2}\partial_{i-1}S)(P)$. Thus, by applying the operator $\partial_{i-j}$, we have $i-j+1 \in (\partial_{i-j} \dots \partial_{i-2}\partial_{i-1}S)(P)$. Since the label $i-j+1$ is fixed by all subsequent promotions $\partial_k$ where $k < i-j$, it follows that $i-j+1 \in (q_{i-1} S)(P)$ as desired.
\end{proof}

We also need the following lemma describing the action of $q_{jk}$ on $\ell_P$ and $\ell_Q$.

\begin{lemma}\label{qjk-action-P*}
    Let $\ell \in \linext(P+Q)$, and let $T(\ell) = (\ell_P, \ell_Q, S(P), S(Q))$. Then 
    \[
    T(q_{jk}\ell) = (q_{m-n,m}(\ell_P), q_{k-m,k-j-n-1}(\ell_Q), (q_{jk}S)(P), (q_{jk}S)(Q)).
    \]
    where $m = |S(P) \cap [k]|$ and $n = |S(P) \cap [j,k]| - 1$.
\end{lemma}

\begin{proof}
    First, we examine the induced action of $q_{jk}$ on $\ell_P$. Consider the action of $q_{k-1} = \partial_0 \partial_1 \dots \partial_{k-1}$ on $\ell$. The description of promotion in Definition~\ref{defn:promo_Stanley} implies that each operator $\partial_i$ in $q_{k-1}$ ($0 \le i \le k-1$) only acts on $(\ell_i)_P$ if $\ell_i^{-1}(1)$ lies in the component $P$ and on $(\ell_i)_Q$ otherwise, where $\ell_i = \partial_{i+1} \dots \partial_{k-1} (\ell)$.
    The number of promotion operators in $q_{k-1}$ that act on the labels of $P$ therefore is equal to the number of linear extensions $\ell_i$ for which $(\ell_i^{-1})(P)$ lies in the component $P$, which by the observation in the previous proof is precisely $m = |S(P) \cap [k]|$. Observe that by Definition~\ref{defn:promo_Stanley}, the action of promotion on $\ell_P$ does not depend on the labels of $\ell$ on $P$, but only their relative order encoded by $\ell_P$.  Hence, we deduce that the action on $\ell$ by $q_{k-1} = \partial_0\partial_1\dots\partial_{k-1}$ is mapped under $T$ to the action on $\ell_P$ by $q_{m-1}$.
    
    Now we consider the action of $q_{k-j}$ on $q_{k-1}\ell$. By the same argument above, the induced action of $q_{k-j}$ on $q_m(\ell_P)$ is equivalent to the action of $q_{n'-1}$ on $q_{m-1}(\ell_P)$, where $n' = |(q_{k-1}S)(P) \cap [k-j+1]|$. If $c \in S(P)$ and $c > k$, by Lemma \ref{qi-action-disjoint-union}, we have $c \in (q_{k-1}S)(P)$, which does not contribute to $n'$. Otherwise, if $c \le k$, by Lemma~\ref{qi-action-disjoint-union}, we have $k-c+1 \in (q_{k-1}S)(P)$. Thus, $n' = |(q_{k-1}S)(P) \cap [k-j+1]| = |S(P) \cap [j, k]|$.
    
    Finally, we act again by $q_{k-1}$ on $q_{k-j}q_{k-1}\ell$. The induced action of $q_{k-1}$ on $q_{n'-1}q_{m-1}\ell_P$ is equivalent to the action of $q_{m'-1}$ on $q_{n'-1}q_{m-1}\ell_P$, where $m' = |(q_{k-j}q_{k-1}S)(P) \cap [k]|$. However, by Lemma \ref{qi-action-disjoint-union}, we have $|(q_{k-j}q_{k-1}S)(P) \cap [k]| = |(q_{k-1}S)(P) \cap [k]| = |S(P) \cap [k]|$.
    Thus, the action of $q_{k-1}$ here is equivalent to the action of $q_{m-1}$ on $q_{n'-1}q_{m-1}\ell_P$, and as the result, the action of $q_{jk}$ on $\ell$ induces the action of $q_{m-1}q_{n'-1}q_{m-1} = q_{m-n,m}$ on $\ell_P$, where $n = n'-1 = |S(P) \cap [j, k]| - 1$.
    
    By symmetry, the induced action on $\ell_Q$ is given by $q_{m'-n',m'}$ where $m' = |S(Q) \cap [k]|$ and $n' = |S(Q) \cap [j,k]| - 1$. Since $|S(P) \cap [k]| + |S(Q) \cap [k]| = k$ and $|S(P) \cap [j,k]| + |S(Q)\cap [j,k]| = k-j+1$, we deduce that $m' = k-m$ and $n' = k-j-n-1$. Hence  
    \[
    T(q_{jk}\ell) = (q_{m-n, m}(\ell_P), q_{k-m,k-j-n-1}(\ell_Q), (q_{jk}S)(P), (q_{jk}S)(Q))
    \]
    where $m = |S(P) \cap [k]|$ and $n = |S(P) \cap [j,k]| - 1$, as desired.
\end{proof}

Now we are ready to prove our main proposition.

\begin{thm} \label{thm:cactus_disj}
    Let $P$ and $Q$ be finite posets.
    If the cactus relations hold in $\mathcal{BK}_P$ and $\mathcal{BK}_Q$, then they hold in $\mathcal{BK}_{P+Q}$.
\end{thm}

\begin{proof}
Let $\ell \in \linext(P+Q)$ be any linear extension of $P+Q$, and let $T(\ell) = (\ell_P, \ell_Q, S(P), S(Q))$. Since the map $T$ is injective, it suffices to show that for all $2 \le i+1 < j < k \le |P+Q|$, $t_i$ and $q_{jk}$ commute with respect to each of $\ell_P, \ell_Q, S(P)$, and $S(Q)$. We consider two general cases.

\textit{Case 1:} Assume without loss of generality that $i, i+1 \in S(P)$. First we show that $t_i$ and $q_{jk}$ commute with respect to $S(P)$ and $S(Q)$. Observe that if $i$ (or $i+1$) is in $S(P)$, by applying Lemma~\ref{qi-action-disjoint-union} three times, we have $i$ (or $i+1$) is also in $(q_{k-1}q_{k-j}q_{k-1}S)(P) = (q_{jk}S)(P)$. It follows that $t_i$ acts as the identity on both $S(P)$ and $(q_{jk}S)(P)$, so clearly $t_i$ and $q_{jk}$ commute.

It remains to show that $t_i$ and $q_{jk}$ commute with respect to $\ell_P$ and $\ell_Q$. Since $i, i+1 \in (q_{jk}S)(P)$, $t_i$ does not affect $\ell_Q$ and $q_{jk}(\ell_Q)$.
It then suffices to show that $t_i$ and $q_{jk}$ commute with respect to $\ell_P$. By Lemma \ref{qjk-action-P*}, the induced actions of $t_iq_{jk}$ and $q_{jk}t_i$ on $\ell_P$ are given by to $t_{i'}q_{m-n,m}$ and $q_{m-n,m}t_{i'}$, respectively, where $m = |S(P) \cap [k]|$ and $n = |S(P) \cap [j,k]|-1$ and $i' = |S(P) \cap [i]|$. Further note that 
\begin{align*}
    (m-n) - i' &= |S(P) \cap [k]| - (|S(P) \cap [j, k]|-1) - |S(P) \cap [i]| \\
    &= |S(P) \cap [i+1, j-1]| + 1 \geq 2
\end{align*}
since $i+1 \in S(P)$, thus $i'+1 < m-n$. Hence, we have $t_{i'}q_{m-n,m}(\ell_P) = q_{m-n,m}t_{i'}(\ell_P)$, following directly from the assumption that $\mathcal{BK}_P$ satisfies the cactus relations.

\textit{Case 2:} Assume without loss of generality that $i \in S(P)$ and $i+1 \in S(Q)$, which configuration we simply denote as $(i, i+1)$. We consider where the labels $i, i+1$ are taken under $t_i$ and $q_{jk}$. When acting on $\ell$ by $q_{jk}t_i$, by the above observation we have $(i,i+1) \xrightarrow{t_i} (i+1,i) \xrightarrow{q_{jk}} (i+1,i)$. On the other hand, acting on $\ell$ by $t_iq_{jk}$ gives $(i, i+1) \xrightarrow{q_{jk}} (i, i+1) \xrightarrow{t_i} (i+1, i)$, which agrees with the action of $q_{jk}t_i$. A consequence is that the set of other labels are preserved by both $t_i q_{jk}$ and $q_{jk}t_i$, and since $t_i$ fixes all other labels, we deduce that $t_i$ and $q_{jk}$ commute with respect to $S(P)$ and $S(Q)$.

By a prior observation, we have $i \in (q_{jk}S)(P)$ and $i+1 \in (q_{jk}S)(Q)$. Thus, $t_i$ acts as the identity on each of $\ell_P, \ell_Q, q_{jk}(\ell_P)$, and $q_{jk}(\ell_Q)$, as swapping elements $i$ and $i+1$ does not change the relative ordering within the label sets. Then clearly $t_i$ and $q_{jk}$ commute with respect to $\ell_P$ and $\ell_Q$.

Hence, we conclude that $t_i q_{jk}(\ell) = q_{jk} t_i (\ell)$ for all $\ell \in \linext(P+Q)$, so the cactus relations hold in $\mathcal{BK}_{P + Q}$.
\end{proof}

The converse of this proposition follows directly from Proposition~\ref{prop:cactus_ideal}. This property then allows us to narrow our focus to studying connected posets whose BK groups satisfy the cactus relations.


\subsubsection{Ordinal sums}

In contrast with disjoint unions, ordinal sums of posets whose BK groups satisfy the cactus relations exhibit a more complicated behavior. In particular, the BK group of an ordinal sum of two posets whose BK groups satisfy the cactus relations may not satisfy the cactus relations.

Let $P$ be an $n$-element poset. We next study the construction $P \mapsto A_m \oplus P$, where $A_m$ is an {\it antichain} of $m$ elements, and ask whether the cactus relations are preserved in their BK groups.

\begin{prop}\label{prop:min-elt-cactus}
If the cactus relations hold in $\mathcal{BK}_P$, then they hold in $\mathcal{BK}_{A_1 \oplus P}$.
\end{prop}

\begin{proof}
    Observe that there is a one-to-one correspondence between the sets of linear extensions of $P$ and $A_1 \oplus P$: a linear extension of $A_1 \oplus P$ is simply a linear extension of $P$ (with labels shifted up by $1$) with the unique minimal element of $A_1 \oplus P$ (namely the singleton in $A_1$) labelled $1$. Thus $t_1$ is the identity. It follows that for any $3\leq i+1<j<k-1 \leq n+1$, the relation $(t_iq_{jk})^2=1$ on $\linext(A_1 \oplus P)$ reduces to a cactus relation on linear extensions of the convex induced subposet $P$ of $A_1 \oplus P$, which holds since the cactus relations hold in $\mathcal{BK}_P$. The only other cases are when $i = 1$ or $k = j+1$. When $i=1$, $(t_1q_{jk})^2 = q_{jk}^2 = 1$ for all $2<j<k\leq n+1$, since $q_{jk}$ is an involution. When $k = j+1$, observe that 
    \[
    q_{jk} = q_{k-1}q_{k-j}q_{k-1} = q_{k-1}q_1q_{k-1} = q_{k-1}t_1q_{k-1} = q_{k-1}^2 = 1,
    \]
    since $q_{k-1}$ is an involution.  Thus $q_{jk}=1$, so that $(t_iq_{jk})^2= t_i^2 = 1$ as desired.
\end{proof}

\begin{prop}
    \label{prop:1min-cac-imply-2min-cac}
    If the cactus relations hold in $\mathcal{BK}_{A_1 \oplus P}$, then they hold in $\mathcal{BK}_{A_2 \oplus P}$.
    Hence if the cactus relations hold in $\mathcal{BK}_P$, then they hold in both $\mathcal{BK}_{A_1 \oplus P}$ and $\mathcal{BK}_{A_2 \oplus P}$.
\end{prop}

\begin{proof}
    Suppose $A_2 = \{p_1\} + \{p_2\}$, and let $P_i$ be the convex induced subposet of $A_2 \oplus P$ of the form $\{p_i\} \oplus P$. 
    By assumption, the cactus relations hold in each $\mathcal{BK}_{P_i}$.
    Similar to the previous argument, observe that a linear extension of $A_2 \oplus P$ is a linear extension of $P$ (with labels shifted up by $2$) with the labels of minimal elements $p_1$ and $p_2$ drawn from $\{1,2\}$. Thus $t_2$ is the identity. For $j \ge 1$, we have
    \[\begin{array}{r l}
         q_j & = t_1(t_2t_1)\ldots(t_{j} t_{j-1} \ldots t_1) = t_1 (t_1) \ldots (t_j \ldots t_3 t_1)  \\[5pt]
         & = t_1^j (t_3)(t_4 t_3) \ldots (t_j t_{j-1} \ldots t_3) = t_1^j q'_{j}
    \end{array}  \]
    where
    \[ q'_{j} = (t_2)(t_3 t_2)(t_4 t_3 t_2) \ldots (t_j t_{j-1} \ldots t_3 t_2) = (t_3)(t_4 t_3) \ldots (t_j t_{j-1} \ldots t_3). \]
    Since $q_1 = t_1$, by definition $q'_1 = 1$. Notice that $t_1$ and $q'_j$ always commute, so for $2 < j < k \le n+2$, we have
    \[ \begin{array}{rl}
         q_{jk} & = q_{k-1}q_{k-j}q_{k-1} = (t_1^{k-1}q'_{k-1})(t_1^{k-j}q'_{k-j})(t_1^{k-1}q'_{k-j})\\[5pt]
         & = t_1^{k-j}(q'_{k-1}q'_{k-j}q'_{k-1}) =: t_1^{k-j} q'_{jk}.
    \end{array}\]
    Since each $q'_j$ commutes with $t_1$, so does $q'_{jk}$.
    
    We show that the cactus relation $(t_i q_{jk})^2 = 1$ holds by considering the following cases:
    
    {\it Case 1: $i = 2$.} We have $(t_2q_{jk})^2 = q_{jk}^2 = 1$, since $q_{jk}$ is an involution.
        
    {\it Case 2: $i = 1$.} We have $t_1q_{jk} = t_1 t_1^{k-j}q'_{jk} = t_1^{k-j}q'_{jk}t_1 = q_{jk}t_1$, so $(t_1q_{jk})^2 = 1$.
    
    {\it Case 3: $i\geq 3$ and $k-j = 1$.} Recall that $q'_{k-j} = q'_1 = 1$, so $q_{jk} = t_1 (q'_{k-1})^2 = t_1$, which obviously commutes with $t_i$ if $i \ge 3$.
    
    {\it Case 4: $i \geq 3$ and $k-j \geq 2$.} We have $(t_iq_{jk})^2 = (t_i t_1^{k-j} q'_{jk})^2 = (t_1^{k-j})^2 (t_iq'_{jk})^2 = (t_iq'_{jk})^2$, so it suffices to show that $(t_i q'_{jk})^2 = 1$ for all $i \ge 3$ and $k-j \ge 2$. This equation reduces to a cactus relation on linear extensions of either $P_1$ or $P_2$, which holds since by assumption the cactus relations hold in each $\mathcal{BK}_{P_i}$.
    
    The second part follows immediately from Proposition~\ref{prop:min-elt-cactus} and the first part.
\end{proof}

Seeing Propositions~\ref{prop:min-elt-cactus} and \ref{prop:1min-cac-imply-2min-cac}, one might think that given any poset $P$ for which the cactus relations hold in $\mathcal{BK}_P$, they continue to hold in $\mathcal{BK}_{A_m \oplus P}$ for all $m \ge 1$.
Unfortunately, this is not the case:

\begin{prop} \label{prop:ord-sum-not-cactus}
    For any non-empty finite poset $P$, some cactus relations do not hold in $\mathcal{BK}_{A_m \oplus P}$ for $m \ge 3$.
\end{prop}

\begin{proof}
    Observe that $A_m \oplus P$ always contains a copy of $A_m \oplus A_1$ as an order ideal. 
    Thus, by Proposition~\ref{prop:cactus_ideal}, it suffices to show that some cactus relations do not hold in $\mathcal{BK}_{A_m \oplus A_1}$ for all $m \ge 3$. 
    When $m=3$, one can check that the cactus relation $(t_1q_{34})^2=1$ does not hold. 
    When $m \geq 4$, we claim that the cactus relation $(t_1q_{m-1,m+1})^2 = 1$ does not hold. By studying the evacuation $q_j$ on linear extensions of the poset $A_m \oplus A_1$ (see, e.g., \cite{Stanley-promotion-evacuation}), we observe that $q_{m-1,m+1}$ swaps the labels $1$ and $m-1$, while $t_1$ swaps the labels $1$ and $2$. These descriptions show that they do not commute.
\end{proof}

While taking the ordinal sum with sufficiently large antichains creates a source of posets whose BK groups do not satisfy the cactus relations, the same operation with sufficiently long \textit{chains} appears to have the exact opposite effect. Let $C_m$ denote a chain of $m$ elements.

\begin{prop}\label{prop:oplus_chain}
    For any $n$-element poset $P$ and for any $m \ge n-3$, the cactus relations hold in $\mathcal{BK}_{C_m \oplus P}$.
\end{prop}

\begin{proof}
    By Proposition~\ref{prop:min-elt-cactus}, it suffices to show that the cactus relations hold in $\mathcal{BK}_{C_{n-3} \oplus P}$. 
    The key observation is that $t_i$ is the identity for all $i \le n-3$, since the element of $P$ labelled by $i$ lies inside the lower chain and thus is comparable to that labelled by $i+1$. Consider the operators $t_i$ and $q_{jk} = q_{k-1}q_{k-j}q_{k-1}$ for $2 \le i+1 < j < k \le 2n-3$. If $i \le n-3$, $t_i=1$ obviously commutes with $q_{jk}$. If $i \geq n-2$, then $j\geq n$, so $k-j \leq (2n-3) - n = n-3$. It follows that $q_{k-j} = t_1(t_2t_1)\ldots(t_{k-j} t_{k-j-1} \ldots t_1) = 1$, therefore $q_{jk} = q_{k-1}^2 = 1$ which obviously commutes with $t_i$.
\end{proof}


\subsection{Minuscule, d-complete, and jeu-de-taquin posets}\label{subsec:cactus_jdt}

Starting with the singleton poset $A_1$, by iterating the constructs of disjoint union and ordinal sum $P \mapsto A_1 \oplus P$, we produce the family of all {\it rooted forest posets}--disjoint unions of trees with a unique minimal element, whose BK groups satisfy the cactus relations by Theorem~\ref{thm:cactus_disj} and Proposition~\ref{prop:min-elt-cactus}. 
Knuth observed a famous hook-length formula for counting the linear extensions of such forest posets \cite{knuth-hook-length}. The original motivation for this finding was the family of Ferrers posets for which the Frame--Robinson--Thrall hook-length formula counts their linear extensions \cite{frt-hook-length}, another family of posets whose BK groups satisfy the cactus relations. Motivated by these two families with hook formulas, Proctor introduced the family of {\it d-complete posets} having such a hook-length formula for their linear extensions; we refer the readers to \cite{ProctorScoppetta} and \cite{KimYoo} for the precise definitions and hook-length formulas. This raises the question of whether the BK groups of d-complete posets satisfy the cactus relations.

Preliminary investigations point to an affirmative answer. Besides the above families, it is not too hard to show that the BK groups of all {\it shifted Ferrers posets}, another motivating subfamily of d-complete posets, satisfy the cactus relations.

\begin{thm}
    The cactus relations hold in $\bkp$ for all shifted Ferrers posets $P$.
\end{thm}

\begin{proof}
   First, we embed standard shifted tableaux, which are in bijection with linear extensions of shifted Ferrers posets, by ``doubling'' them into {\it shift-symmetric} column-strict tableaux \cite{SaganShifted}; see Example~\ref{ex:shift_symm_tabl}. Since the content of a standard shifted tableau strictly increases across the column and row, the classical BK moves $t_i$ act the same on the resulting shift-symmetric column-strict tableau restricting to the embedded standard shifted part as on the standard shifted tableau itself. That is, the action of $t_i$ on the standard shifted tableau commute with the action of $t_i$ on the shift-symmetric tableau, where BK moves satisfy the cactus relations. Therefore, the cactus relations hold in $\mathcal{BK}_P$ for posets $P$ of shifted Young diagram shape, concluding the proof.
\end{proof}

\begin{example}\label{ex:shift_symm_tabl}
Below is an example construction of a shift-symmetric column-strict tableau from a standard shifted tableau.
     \[ S = \;
        \begin{ytableau} 
            1 & 2 & 3 & 4\\ 
            \none & 5 & 6 & 7 \\
            \none & \none & 8 & 9
            \end{ytableau} \; \quad \Rightarrow \quad
     T = \;
        \begin{ytableau} 
            *(yellow) 1 & 1 & 2 & 3 & 4\\ 
            *(yellow) 2 & *(yellow) 5 & 5 & 6 & 7 \\
            *(yellow) 3 & *(yellow) 6 & *(yellow) 8 & 8 & 9 \\
            *(yellow) 4 & *(yellow) 7 & *(yellow) 9 & \none & \none
            \end{ytableau} \; .
    \]
\end{example}

Another closely related family of posets is the \textit{minuscule posets}, which first appeared in Lie theory.
Proctor's d-complete posets contain all order ideals within minuscule posets as a motivating special case. Minuscule posets are classified into three infinite families and two exceptional cases (see, e.g., \cite{stembridge,rush-shi}). The infinite families of minuscule posets are the rectangular Ferrers posets (whose order ideals contain all Ferrers posets), the triangular shifted Ferrers posets (whose order ideals contain all shifted Ferrers posets), and ordinal sums of antichains $A_1^{\oplus n} \oplus A_2 \oplus A_1^{\oplus n}$; 
all of whose BK groups satisfy the cactus relations based on our results so far. By checking the cactus relations in the BK groups of the remaining exceptional minuscule posets using {\it SageMath}, we have verified the following:

\begin{thm}
    The cactus relations hold in $\bkp$ for all minuscule posets $P$.
\end{thm}

Based on the existing evidence, we make the following conjecture.

\begin{conjecture}\label{conj:d-complete-cactus}
    The cactus relations hold in $\bkp$ for all d-complete posets $P$.
\end{conjecture}

This has been verified  by {\it SageMath} for all d-complete posets of size at most $9$. 

\begin{remark}
    In \cite{proctor2009d}, Proctor also studied a large family of posets called \textit{jeu-de-taquin posets}, which he showed includes all d-complete posets. Informally, jeu-de-taquin posets are posets for which Sch\"utzenberger's jeu-de-taquin on linear extensions of posets, as discussed in \cite{Stanley-promotion-evacuation}, has an extra confluence property. 
    
    For such posets, we could show that $q_{jk}$ only permutes the labels $j,j+1,\dots,k$ when acting on a linear extension. Thus, it is more probable for $t_i$ and $q_{jk}$ to commute because they affect disjoint subsets of the labels. Hence, one might hope that the BK groups of jeu-de-taquin posets satisfy the cactus relations. However, this is not true; Figure~\ref{fig: min-counterexample} shows a counterexample, where the cactus relation $(t_3q_{59})^2 = 1$ does not hold for this linear extension whose underlying poset is jeu-de-taquin.

    \begin{figure}[!ht]
        \centering

            \resizebox{0.3\textwidth}{!}{
                \begin{tikzpicture}
                
                \draw[black, thick] (-1,0.5) -- (-1,2);
                \draw[black, thick] (-1,2) -- (1,3);
                \draw[black, thick] (1,7.5) -- (1,3);
                \draw[black, thick] (-1,2) -- (-3,3.75);
                \draw[black, thick] (-1,6) -- (-3,3.75);
                \draw[black, thick] (1,4.5) -- (-1,6);
                \draw[black, thick] (1,4.5) -- (3,6);
                \draw[black, thick] (1,7.5) -- (-1,6);
                \draw[black, thick] (1,7.5) -- (3,6);
                
                \filldraw[black] (1,7.5) circle (2pt);
                \filldraw[black] (1,6) circle (2pt);
                \filldraw[black] (1,4.5) circle (2pt);
                \filldraw[black] (1,3) circle (2pt);
                \filldraw[black] (-1,2) circle (2pt);
                \filldraw[black] (-1,0.5) circle (2pt);
                \filldraw[black] (-1,6) circle (2pt);
                \filldraw[black] (3,6) circle (2pt);
                \filldraw[black] (-3,3.75) circle (2pt);
                
                \draw (-1,0) node[anchor=center] {\huge $1$};
                \draw (-1,2.5) node[anchor=center] {\huge $2$};
                \draw (1,2.5) node[anchor=center] {\huge $3$};
                \draw (-3,3.25) node[anchor=center] {\huge $4$};
                \draw (1.25,4.25) node[anchor=center] {\huge $5$};
                \draw (3.5,6) node[anchor=center] {\huge $6$};
                \draw (-1.25,6.25) node[anchor=center] {\huge $7$};
                \draw (1.5,6) node[anchor=center] {\huge $8$};
                \draw (1,8) node[anchor=center] {\huge $9$};
                
                \end{tikzpicture}
            }
        \caption{The smallest jeu-de-taquin poset whose BK group does not satisfy the cactus relations.}
        \label{fig: min-counterexample}
    \end{figure}
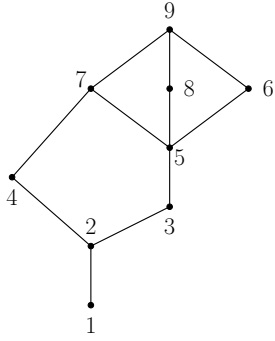
\end{remark}


\section{Posets with the full symmetric BK group}\label{sec:symm}

In this section, we will study the group $\bkp$ as a subgroup of the symmetric group on the set $\linext(P)$. In particular, we focus on understanding posets $P$ for which $\bkp = \symm_{\linext(P)}$.


\subsection{Duals and ordinal sums}\label{subsec:symm_prop}

As in Section~\ref{subsec:cactus}, we explore properties of posets whose BK groups equal the symmetric group on the set of their linear extensions, particularly their behavior under poset operations such as taking duals and ordinal sums. Let $P$ and $Q$ be finite posets, and let $P^*$ denote the {\it dual} or {\it opposite} poset to $P$.

\begin{prop}\label{prop:BKP_dual}
One has a group isomorphism
    $\bkp \cong \mathcal{BK}_{P^*}$. 
    
Hence $\mathcal{BK}_P \cong \symm_{\linext(P)}$ if and only if $\mathcal{BK}_{P^*} \cong \symm_{\linext(P^*)}$.
\end{prop}

\begin{proof}
One has a bijection $\linext(P) \to \linext(P^*)$ sending $\ell \mapsto \ell'$ where $\ell'(v) = |P|+1-\ell(v)$ for all $v \in P$. This induces an isomorphism of symmetric groups $\symm_{\linext(P)} \rightarrow \symm_{\linext(P^*)}$, which then  restricts to an isomorphism 
$\bkp \to \mathcal{BK}_{P^*}$ sending $t_i \mapsto t_{|P|-i}$.
\end{proof}

Likewise, the BK group of the ordinal sum of two posets has a very nice description in terms of the BK groups of its components.

\begin{prop}\label{prop:BKP_osum}
    $\mathcal{BK}_{P \oplus Q} \cong \bkp \times \bkq$.
\end{prop}

\begin{proof}
    Let
    \begin{align*}
        \lambda : \linext(P \oplus Q) &\to \linext(P) \times \linext(Q) \\
        \ell &\mapsto (\ell_1, \ell_2),
    \end{align*}
    where $\ell_1$ is $\ell$ restricted to $P$ and $\ell_2$ is $\ell - |P|$ restricted to $Q$. We have that $\lambda$ is a bijection. Now let
    \begin{align*}
        \psi : \mathcal{BK}_{P \oplus Q} &\to \bkp \times \bkq \\
        t_i &\mapsto t'_i,
    \end{align*}
    where
    \[
        t'_i = \begin{cases}
            (t_i, 1) &\quad\text{if} \ i \le |P|-1 \\
            (1, 1) &\quad\text{if} \ i = |P| \\
            (1, t_{i-|P|}) &\quad\text{if} \ i \ge |P|+1.
        \end{cases}
    \]
    We have that $\psi$ is an isomorphism of permutation groups as desired.
\end{proof}

This property leads to some useful tools for constructing and identifying posets with the full symmetric BK groups that arise from ordinal sums.

\begin{cor}\label{cor:symm_min_max}
    Let $P$ be a finite poset.
    Then $\mathcal{BK}_{P} \cong \symm_{\linext(P)}$ if and only if $\mathcal{BK}_{C_a \oplus P \oplus C_b} \cong \symm_{\linext(C_a \oplus P \oplus C_b)}$ for any chains $C_a$ and $C_b$.
\end{cor}

\begin{proof}
This corollary follows directly from the previous proposition and the fact that for any chain $C_a$, the set $\linext(C_a)$ is a singleton, so its BK group is trivial.
\end{proof}

\begin{cor}\label{cor:symm_osum}
    If $P$ and $Q$ are posets such that $\mathcal{BK}_{P \oplus Q} \cong \symm_{\linext(P \oplus Q)}$, then at least one of $P$ or $Q$ is a chain.
    As a consequence, $\mathcal{BK}_P \cong \symm_{\linext(P)}$ and $\mathcal{BK}_Q \cong \symm_{\linext(Q)}$.
\end{cor}

\begin{proof}
    By assumption, $\mathcal{BK}_{P \oplus Q}$ is a symmetric group. 
    By Proposition~\ref{prop:BKP_osum},  $\mathcal{BK}_{P \oplus Q}$ can be written as a direct product of the ordinal summands' BK groups, i.e., $\mathcal{BK}_{P \oplus Q} ~\cong~ \bkp \times \bkq$. Then at least one of the factors $\bkp$ and $\bkq$ is the trivial group. This happens when at least one of the posets $P$ and $Q$ has exactly one linear extension and hence is a chain. The second part follows immediately from the previous corollary.
\end{proof}

Corollaries~\ref{cor:symm_min_max} and ~\ref{cor:symm_osum} are especially useful in identifying posets whose BK groups are the full symmetric group, as they allow us to turn our attention to posets with neither a unique maximal nor minimal element.


\subsection{Primitive BK groups and disjoint unions of posets}\label{subsec:BKP_prim}

The goal of this subsection is to prove a classification theorem for disconnected posets with the full symmetric BK groups. En route to this result, we will first study a generalization of these posets. Recall that the action of a group $G$ on a set $X$ is \emph{primitive} if there is no partition of $X$ that is preserved by $G$ other than the trivial partitions (i.e., the finest and coarsest partitions). Evidently, $\symm_{\linext(P)}$ is a primitive permutation group of $\linext(P)$. In this subsection, we turn our attention to posets whose BK groups are primitive.

\begin{lemma}\label{lem:BKP_disj_chains}
    If an $n$-element poset $P$ is a disjoint union of two or more non-empty chains, then $\bkp \cong \symm_n$.
\end{lemma}

\begin{proof}
    By Proposition~\ref{prop:rel_braid}, $\bkp$ is isomorphic to a quotient of $\symm_n$. Thus, it suffices to show that the group is not trivial, is not $C_2$ for $n \ge 3$, and is not $\symm_3$ for $n=4$. Each of these is easily checked.
\end{proof}

\begin{lemma} \label{lem-stab-chains}
    Let $P = C_{n_1} + C_{n_2} + \dots + C_{n_r}$. Then, $\Stab(\ell) \cong \symm_{n_1} \times \symm_{n_2} \times \dots \times \symm_{n_r}$ for any $\ell \in \linext(P)$.
\end{lemma}

\begin{proof}
    We have $\Stab(\ell_1) \cong \Stab(\ell_2)$ for any $\ell_1, \ell_2 \in \linext(P)$, so it suffices to consider only one linear extension. Let $\ell$ be the linear extension that assigns consecutive labels to each chain beginning with $C_{n_1}$. Then, $t_i \in \Stab(\ell)$ if and only if $i \neq \sum_{j=1}^k n_j$ for all $1 \le k \le r-1$. Recall that by Lemma~\ref{lem:BKP_disj_chains}, $t_1, \dots, t_{n-1}$ yield a presentation for the symmetric group $\symm_n$, where
    \[
        n = \sum_{j=1}^r n_j.
    \]
    Thus, the group generated by $\{t_i \mid i \neq \sum_{j=1}^k n_j \ \text{for all} \ 1 \le k \le r-1\}$ is isomorphic to $\symm_{n_1} \times \symm_{n_2} \times \dots \times \symm_{n_r}$. It remains to show that $\symm_{n_1} \times \symm_{n_2} \times \dots \times \symm_{n_r}$ exhausts $\Stab(\ell)$, which follows from the following calculation:
    \begin{align*}
        |\linext(P)| \cdot |\Stab(\ell)| &= |\bkp| \\
        \frac{n!}{n_1! \cdots n_r!} \cdot |\Stab(\ell)| &= n! \\
        |\Stab(\ell)| &= n_1! \cdots n_r! = |\symm_{n_1} \times \symm_{n_2} \times \dots \times \symm_{n_r}|.
    \end{align*}
\end{proof}

\begin{lemma}\label{lemma:not-chains-not-prim}
    If $P = P_1 + P_2$, where $|P_1|, |P_2| \ge 1$ and $P_1$ is not a chain, then $\mathcal{BK}_P$ is not a primitive permutation group.
\end{lemma}
\begin{proof}
    Given $\ell \in \linext(P)$, assign $\ell$ the set $\ell(P_1)$. This induces a partition on $\linext(P)$. The conditions of the lemma imply that the partition is nontrivial. It is easy to check that each $t_i$ preserves the partition.
\end{proof}

\begin{thm} \label{thm:prim-disconn}
    If $P$ is a disconnected poset, then $\mathcal{BK}_P$ is a primitive permutation group if and only if $P = C_{n_1} + C_{n_2}$, where $n_1 \neq n_2$ or $n_1=n_2=1$.
\end{thm}

\begin{proof}
    First, if one component of $P$ is not a chain, then $\mathcal{BK}_P$ is not primitive by Lemma~\ref{lemma:not-chains-not-prim}. If $P = C_{n_1} + C_{n_2} + \dots + C_{n_r}$, then $\bkp \cong \symm_n$ by Lemma~\ref{lem:BKP_disj_chains} and $\Stab(\ell) \cong \symm_{n_1} \times \symm_{n_2} \times \dots \times \symm_{n_r}$ for any $\ell \in \linext(P)$ by Lemma~\ref{lem-stab-chains}. We have that $\symm_{n_1} \times \symm_{n_2} \times \dots \times \symm_{n_r}$ is a maximal proper subgroup of $\symm_n$ if and only if $r=2$ and either $n_1 \neq n_2$ or $n_1 = n_2 = 1$. Since a permutation group is primitive if and only if the stabilizer of any element is a maximal subgroup, this completes the proof.
\end{proof}

The first application of this theorem is a classification of {\it series-parallel} posets--posets that can be constructed from the singleton $A_1$ using the ordinal sum and disjoint union operations--with primitive BK groups.

\begin{cor}
    If $P$ is a non-empty series-parallel poset, then $\mathcal{BK}_P$ is a primitive permutation group if and only if $P = C_{n_1} \oplus (C_{n_2} + C_{n_3}) \oplus C_{n_4}$, where $|n_2| \neq |n_3|$ or $|n_2| = |n_3| = 1$.
\end{cor}
\begin{proof}
    Every poset of the stated form is series-parallel with a primitive BK group by Theorem~\ref{thm:prim-disconn}. Conversely, let $P$ be a non-empty series-parallel poset whose BK group is a primitive permutation group. If $P = P_1 + P_2$, where $|P_1|, |P_2| \ge 1$, then $P = C_{n_2} + C_{n_3}$ with the stated conditions by Theorem~\ref{thm:prim-disconn}. If $P = P_1 \oplus P_2$, where $|P_1|, |P_2| \ge 1$, then $P_1 = C_{n_1}$ or $P_2 = C_{n_4}$ by Proposition~\ref{prop:BKP_osum}. Thus, $P$ is of the desired form.
\end{proof}


We now return to disconnected posets whose BK groups are the full symmetric group. The proof of our classification theorem in this case follows a similar logic as that of Theorem~\ref{thm:prim-disconn}.

\begin{thm}\label{thm:symm_disconn}
    If $P$ is a disconnected poset, then $\mathcal{BK}_P \cong \symm_{\linext(P)}$ if and only if $P = C_{n_1} + A_1$.
\end{thm}
\begin{proof}
    First, if one component of $P$ is not a chain, then $\mathcal{BK}_P \ncong \symm_{\linext(P)}$ by Lemma~\ref{lemma:not-chains-not-prim}. If $P = C_{n_1} + C_{n_2} + \dots + C_{n_r}$, then $\bkp \cong \symm_n$ by Lemma~\ref{lem:BKP_disj_chains}. Thus, $\mathcal{BK}_P \cong \symm_{\linext(P)}$ if and only if $|\linext(P)| = n$, which happens if and only if $P$ is of the stated form.
\end{proof}

\begin{cor}\label{cor:symm_seri_paral}
    If $P$ is a non-empty series-parallel poset $P$, then $\mathcal{BK}_P \cong \symm_{\linext(P)}$ if and only if $P = C_{n_1} \oplus (C_{n_2} + A_1) \oplus C_{n_3}$.
\end{cor}
\begin{proof}
    Every poset $P$ of the form $C_{n_1} \oplus (C_{n_2} + A_1) \oplus C_{n_3}$ is series-parallel and satisfies $\mathcal{BK}_P  \cong \symm_{\linext(P)}$ by Theorem~\ref{thm:symm_disconn}. Conversely, let $P$ be a non-empty series-parallel poset such that $\mathcal{BK}_P  \cong \symm_{\linext(P)}$. If $P = P_1 + P_2$, where $|P_1|, |P_2| \ge 1$, then $P = C_{n_2} + A_1$ by Theorem~\ref{thm:symm_disconn}. If $P = P_1 \oplus P_2$, where $|P_1|, |P_2| \ge 1$, then $P_1 = C_{n_1}$ or $P_2 = C_{n_4}$ by Proposition~\ref{prop:BKP_osum}. Thus, $P$ is of the desired form.
\end{proof}


\subsection{Families of connected posets}\label{subsec:symm_conn}

We next study non-series-parallel posets whose BK groups are the full symmetric group on the set of linear extensions. Theorem~\ref{thm:symm_disconn} allows us to restrict our study to those which are connected. The following class of posets presents a suggestive starting point:

\begin{definition}\label{defn:N-poset}
    For integers $a,b,c \ge 1$, the poset $N_{a,b,c}$ is defined to be a poset of $a+b+c$ elements with the relations
    \[ v_1 < v_2 < \dots < v_{a+1} > v_{a+2} > \dots > v_{a+b+1} < v_{a+b+2} < \dots < v_{a+b+c+1}. \]
\end{definition}

\begin{figure}[htp]
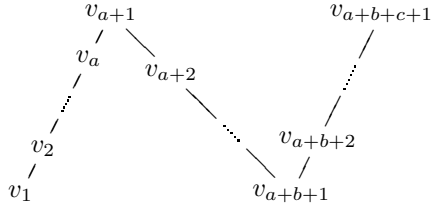

    \centering
    \[ 
	\xy
		{\ar@{-} (0,0)*+{v_1}; (3,6)*+{v_2} };
		{\ar@{-} (3.5,7)*+{}; (5.5,11)*+{} };
		{\ar@{..} (5,10)*+{}; (7,14)*+{} };
		{\ar@{-} (6.5,13)*+{}; (8.5,17)*+{} };
		{\ar@{-} (9,18)*+{v_{a}}; (12,24)*+{v_{a+1}} };
		{\ar@{-} (20,16)*+{v_{a+2}}; (13,23)*+{} };
		{\ar@{-} (27,9)*+{}; (21,15)*+{} };
		{\ar@{..} (30,6)*+{}; (26,10)*+{} };
		{\ar@{-} (36,0)*+{v_{a+b+1}}; (29,7)*+{} };
		{\ar@{-} (36.5,1)*+{}; (39.5,7)*+{v_{a+b+2}} };
		{\ar@{-} (40,8)*+{}; (43,14)*+{} };
		{\ar@{..} (42.5,13)*+{}; (45,18)*+{} };
		{\ar@{-} (44.5,17)*+{}; (48,24)*+{v_{a+b+c+1}} };
	\endxy
\]
    \caption{The poset $N_{a,b,c}$.}
    \label{fig:N-poset}
\end{figure}

In particular, every non-empty non-series-parallel poset contains a poset $N_{1,b,1}$ as a convex induced subposet.

\begin{prop}\label{prop:N1b1}
    $\mathcal{BK}_{N_{1,b,1}} \cong \symm_{\linext(N_{1,b,1})}$ for all $b \ge 1$.
\end{prop}

\begin{proof}
    Let $n = b+3$, so $N_{1,b,1} = v_1 < v_2 > v_3 > \ldots > v_{n-2} > v_{n-1} < v_n$.

    For $n \in \{4, 5\}$, the theorem is verified by a calculation. So suppose $n \ge 6$.
    
    For $i,j \in [n]$, there is at most one linear extension $\ell$ with $\ell(v_1) = i$ and $\ell(v_n) = j$. If it exists, we denote it $\ell_{i,j}$.
    
    First, we claim $(t_2 t_3)^3$ and $(t_{n-3} t_{n-2})^3$ are the transpositions $(\ell_{3,1}\; \ell_{4,1})$ and $(\ell_{n,n-3}\; \ell_{n,n-2})$ respectively. This follows from the characterization of when the braid relations hold.
    
    Now, it suffices to prove that $\mathcal{BK}_{N_{1,b,1}}$ is 2-transitive, which is equivalent to its point-stabilizers' being transitive.
    Consider $\Stab(\ell_{n,1})$. Note that $t_2, t_3, \dots, t_{n-2} \in \Stab(\ell_{n,1})$. Hence, $\Stab(\ell_{n,1})$ is transitive on each of $\{\ell_{i,j} \mid i \le n-1, \ j \ge 2\}$, $\{\ell_{i,j} \mid i \le n-1, \ j = 1\}$, and $\{\ell_{i,j} \mid i = n, \ j \ge 2\}$. It remains to observe that
    \begin{align*}
    (\ell_{3,2}\; \ell_{n,2}) &= (t_1 t_{n-1} t_{n-2} \dots t_4) (t_2 t_3)^3 (t_4 \dots t_{n-2} t_{n-1} t_1) \in \Stab(\ell_{n,1}) \\
    (\ell_{n-1, 1}\; \ell_{n-1,n-2}) &= (t_{n-1} t_1 t_2 \dots t_{n-4}) (t_{n-3} t_{n-2})^3 (t_{n-4} \dots t_2 t_1 t_{n-1}) \in \Stab(\ell_{n,1}),
    \end{align*}
    so $\Stab(\ell_{n,1})$ is transitive.
\end{proof}

\begin{prop}\label{prop:N11c}
    $\mathcal{BK}_{N_{1,1,c}} \cong \symm_{\linext(N_{1,1,c})}$ and $\mathcal{BK}_{N_{a,1,1}} \cong \symm_{\linext(N_{a,1,1})}$ for all $a, c \ge 1$.
\end{prop}
\begin{proof}
    By duality, it suffices to prove the result for $N_{1,1,c}$. Let $n = c+3$, so $N_{1,1,c} = v_1 < v_2 > v_3 < v_4 < \dots < v_n$.

    For $i \in [n-1]$ and $j \in [\max\{3,i+1\},n]$, there is exactly one linear extension $\ell$ with $\ell(v_1)=i$ and $\ell(v_n) = j$, which we denote $\ell_{i,j}$.

    First, we claim $(t_2 t_3)^3= (\ell_{1,3}\; \ell_{1,4})$. This follows from the characterization of when the braid relations hold.

    Now, it suffices to prove that $\mathcal{BK}_{N_{1,1,c}}$ is 2-transitive, which is equivalent to its point-stabilizers' being transitive. Consider $\Stab(\ell_{n-1,n})$. Note that $t_1$, $t_2, \dots, t_{n-3}$, $t_{n-1} \in \Stab(\ell_{n-1,n})$. Hence, $\Stab(\ell_{n-1,n})$ is transitive on $\{\ell_{i,j} \mid j \le n-2\}$ and $\{\ell_{i,j} \mid j \ge n-1, \ i \le n-1\}$. It remains to observe that
    \[
    (\ell_{1,3}\; \ell_{1,n-1}) = (t_{n-2} t_{n-3} \dots t_4) (t_2 t_3)^3 (t_4 \dots t_{n-3} t_{n-2}) \in \Stab(\ell_{n-1,n}),
    \]
    so $\Stab(\ell_{n-1,n})$ is transitive.
\end{proof}

By using the same proof strategy, we are able to prove that the BK groups of the following posets are the full symmetric group on the set of linear extensions:

\begin{figure}[htp]
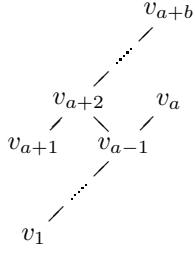

    \centering
    \[ 
	\xy
		{\ar@{-} (6,6)*+{v_1}; (11,11)*+{} };
		{\ar@{..} (10,10)*+{}; (14,14)*+{} };
		{\ar@{-} (13,13)*+{}; (17,17)*+{} };
		{\ar@{-} (18,18)*+{v_{a-1}}; (24,24)*+{v_{a}} };
		{\ar@{-} (17,19)*+{}; (13,23)*+{} };
		{\ar@{-} (6,18)*+{v_{a+1}}; (12,24)*+{v_{a+2}} };
		{\ar@{-} (13,25)*+{}; (17,29)*+{} };
		{\ar@{..} (16,28)*+{}; (20,32)*+{} };
		{\ar@{-} (19,31)*+{}; (24,36)*+{v_{a+b}} };
	\endxy
\]
    \caption{The poset $M_{a,b}$.}
    \label{fig:M-poset}
\end{figure}

\begin{prop}\label{prop:N-hanging-chains}
    For $a,b \geq 2$, let $M_{a,b}$ be a poset on $a+b$ elements with the relations $v_1 < v_2 < \dots < v_a$, $v_{a+1} < v_{a+2} < \dots < v_{a+b}$, and $v_{a-1} < v_{a+2}$ (see Figure~\ref{fig:M-poset}). Then $\mathcal{BK}_{M_{a,b}} \cong \symm_{\linext(M_{a,b}})$.
\end{prop}

\begin{proof}
    First, let $\ell^*$ be the linear extension such that $\ell^*(v_i) = i$. Notice that
    \[
    t_1, t_2, \dots, \widehat{t_{a}}, \dots, t_{a+b-1} \in \Stab(\ell^*).
    \]
    Further note that $t_{a} \notin \Stab(\ell)$ if and only if $\ell \in \{\ell^*, t_a(\ell^*)\}$, i.e., $t_{a}$ acts as the transposition $(\ell^*\; t_a(\ell^*))$ on $\linext(M_{a,b})$.
    
    Thus, it suffices to show that $\Stab(\ell^*)$ is transitive on $\linext(M_{a,b}) \setminus \{\ell^*\}$. This follows from the observations above and the fact that BK groups are always transitive (Proposition~\ref{prop:BKP_transitive}).
\end{proof}

One way to generalize the poset $M_{a,b}$ is to replace the covering relation $v_{a-1} < v_{a+2}$ with another relation $x < y$, where $x$ and $y$ are drawn from two disjoint chains. However, in general these posets do not have the full symmetric BK groups; for example, we found a counterexample in the poset defined by the relations $v_1 < v_2 < v_3 < v_4$, $v_5 < v_6 < v_7 < v_8$, and $v_2 < v_7$.

\subsubsection{Conjectural families and computational results}\label{subsec:symm_conj}

A natural conjectural family of posets $P$ for which $\mathcal{BK}_P = \symm_{\linext(P)}$ follows from our investigation above.

\begin{conjecture}\label{conj:Nabc}
    $\mathcal{BK}_{N_{a,b,c}} \cong \symm_{\linext(N_{a,b,c})}$ for all $a,b,c \ge 1$.
\end{conjecture}

Besides the subfamilies proved in Propositions~\ref{prop:N1b1} and~\ref{prop:N11c}, our preliminary computation by {\it SageMath} indicates that the conjecture holds for all $1 \le a,b,c \le 4$. As a remark, the proofs of Propositions~\ref{prop:N1b1}, ~\ref{prop:N11c}, and ~\ref{prop:N-hanging-chains} rely on an observational analysis of the linear extension graphs of these posets. As the parameters $(a, b, c)$ grow, the linear extension graph of the poset $N_{a,b,c}$ becomes significantly more complex; for example, its dimension increases as parameters $a$ and $c$ grow, and it is no longer apparent that we can generate a transposition on $\linext(N_{a,b,c})$ from BK moves. A different proof strategy may be required for the general case.

Another family of posets that is closely related to $N_{a,b,c}$ is also conjectured to have the full symmetric BK groups.

\begin{conjecture}\label{conj:Zn}
The zigzag-poset $Z_n$, defined to be a poset of $n$ elements with the relations
    \[ v_1 < v_2 > v_3 < v_4 > \ldots > v_{n-3} < v_{n-2} > v_{n-1} < v_n, \]
satisfies $\mathcal{BK}_{Z_n} \cong \symm_{\linext(Z_n)}$ when $n$ is even.
\end{conjecture}

\begin{figure}[htp]
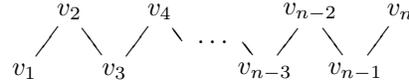

    \centering
    \[ 
	\xy
		{\ar@{-} (0,0)*+{v_1}; (6,8)*+{v_2} };
		{\ar@{-} (11.1,1.2)*+{}; (6.9,6.8)*+{} };
		{\ar@{-} (12,0)*+{v_3}; (18,8)*+{v_4} };
		{\ar@{-} (21.75,3)*+{}; (18.9,6.8)*+{} };
		{(25,4)*+{\ldots}};
		{\ar@{-} (31.1,1.2)*+{}; (28.25,5)*+{} };
		{\ar@{-} (32,0)*+{v_{n-3}}; (38,8)*+{v_{n-2}} };
		{\ar@{-} (43.1,1.2)*+{}; (38.9,6.8)*+{} };
		{\ar@{-} (44,0)*+{v_{n-1}}; (50,8)*+{v_n} };
	\endxy
\]
    \caption{The zigzag-poset $Z_n$ when $n$ is even.}
    \label{fig:Z-poset}
\end{figure}

This conjecture is computationally checked by {\it SageMath} for all even $n \le 10$. When $n$ is odd, we found small counterexamples in $Z_5, Z_7$, and $Z_9$. A potential common generalization of Conjectures~\ref{conj:Nabc} and \ref{conj:Zn} concerns the zig-zag posets of the form $Z_{a_1,\dots,a_n}$ for odd $n$, where $a_i$ is the length of the $i^{\mathrm{th}}$ segment; however, we found a counterexample in $Z_{1,2,2,2,2}$ whose BK group is not the full symmetric group. Finally, we remark that the linear extension graphs of zig-zag posets are very complicated and high-dimensional, for which our observational analysis used in the proof of Proposition~\ref{prop:N1b1} is not suitable.

Our final conjecture involves the Ferrers posets, whose BK groups--in contrast with the fact that they always satisfy the cactus relations \cite{CGP}--do not always equal the full symmetric group on the set of linear extensions.

\begin{conjecture}\label{conj:ferrers}
    The families of Ferrers posets $F_\lambda$ satisfy $\mathcal{BK}_{F_{\lambda}} \cong \symm_{\linext(F_{\lambda})}$ for the following partitions $\lambda$:
    \begin{enumerate}
        \item $\lambda=(n, n-2)$ for all $n$;
        \item $\lambda=(n,3)$ for $n \not\equiv 2 \pmod 4$; and
        \item $\lambda=(n,2,2)$ for $n \not\equiv 0 \pmod 4$.
    \end{enumerate}
\end{conjecture}

Observe that for any partition $\lambda$ and its conjugate $\lambda^t$, $F_{\lambda} \cong F_{\lambda^t}$ are isomorphic posets, so Conjecture~\ref{conj:ferrers} in fact claims that the same property holds for the conjugates of all listed families of Ferrers posets.
These conjectural classes have been computationally checked by {\it SageMath} for (1) $n \leq 10$, (2) $n \leq 18$, and (3) $n \leq 16$. The data set for this analysis was generously provided by J.~Kamnitzer.


\section{Size of BK groups}\label{sec:stab}

In this final section, we explore a notion of size of the group $\bkp$ relative to the total number of linear extensions of a given poset $P$. 


\begin{definition}
     For a poset $P$, its \emph{stabilizer size} $\mathfrak{s}_P$ is defined to be
    \[ \mathfrak{s}_P = \frac{|\bkp|}{|\linext(P)|}. \]
\end{definition}

Our naming is suggestive: Since the subgroup $\bkp$ of $\symm_{\linext(P)}$ is transitive, by the orbit-stabilizer theorem, $\mathfrak{s}_P$ is the size of the stabilizer of any linear extension of $P$. In particular, $\mathfrak{s}_P$ is a positive integer.

Part (1) of Proposition~\ref{prop:BKP_osum} 
and the fact that $|\linext(P \oplus Q)|=|\linext(P)| \cdot |\linext(Q)|$ immediately imply the following.

\begin{prop} \label{prop:stab_size_osum}
    For any posets $P$ and $Q$, one has $\mathfrak{s}_{P \oplus Q} = \mathfrak{s}_P \mathfrak{s}_Q.$
\end{prop}

We now consider what integers are possible values of $\mathfrak{s}_P$. If
\[
    P = C_{n_1} + C_{n_2} + \dots + C_{n_k},
\]
where $C_{n_i}$ is a chain of size $n_i$, then $\mathfrak{s}_P = n_1! n_2! \dots n_k!$ by Lemma~\ref{lem-stab-chains}. In particular, any integer that can be expressed as a product of factorials (a \emph{Jordan--P\'olya number}) is a possible value of $\mathfrak{s}_P$. But this condition is not necessary. For example, $(A_3 \oplus A_1 \oplus A_1) + A_1$ has $\mathfrak{s}_P = 466560$. Thus, we turn to establishing some necessary conditions on possible values of $\mathfrak{s}_P$.

Recall that a poset is indecomposable if it is not an ordinal sum of two or more non-empty posets.

\begin{lemma} \label{lem-s4}
    If $P$ is indecomposable and not a disjoint union of chains, then $\symm_4 \le \Stab(\ell)$.
\end{lemma}
\begin{proof}
    Since $P$ is not a disjoint union of chains, $P$ either has a V-shaped or inverted V-shaped induced subposet. By duality, we may assume $P$ has an inverted V-shaped induced subposet; that is, there are $u,v,w \in P$ with the only relations among them being $v < u$ and $w < u$. Since $P$ is indecomposable, there is an immediate child of $u$ incomparable to some $x$, where $x$ is incomparable with $u$. Without loss of generality, assume this immediate child is $w$. If $x$ and $v$ are incomparable, then $u$, $v$, $w$, and $x$ form a convex induced subposet of $P$ with $\Stab(\ell) \cong \mathbb{Z}/2\mathbb{Z} \times \symm_4$. If $x > v$, then $u$, $v$, $w$, and $x$ form a convex induced subposet of $P$ with $\Stab(\ell) \cong \symm_4$. In either case, the proposition follows by Corollary~\ref{cor:BKP_cvx_ind}.
\end{proof}

\begin{thm}
    If $\mathfrak{s}_P \neq 6,12,36,2^n$, then either $(\mathbb{Z}/2\mathbb{Z})^2 \times \symm_3 \le \Stab(\ell)$ or $\symm_4 \le \Stab(\ell)$. In particular, $\mathfrak{s}_P = 6,12,36,2^n$ or $24 \mid \mathfrak{s}_P$.
\end{thm}
\begin{proof}
    We may assume that $P$ is indecomposable. If $P$ is a disjoint union of chains, the theorem follows from Lemma~\ref{lem-stab-chains}. Otherwise, the theorem follows from Lemma~\ref{lem-s4}.
\end{proof}

If $24 \nmid \mathfrak{s}_P$, then by Lemma~\ref{lem-s4}
\[
    P = P_1 \oplus P_2 \oplus \dots \oplus P_k,
\]
where $P_i$ is a disjoint union of chains. Then, by Proposition~\ref{prop:stab_size_osum},
\[
    \mathfrak{s}_P = \mathfrak{s}_{P_1} \mathfrak{s}_{P_2} \dots \mathfrak{s}_{P_k},
\]
where each $\mathfrak{s}_{P_i}$ can be calculated with Lemma~\ref{lem-stab-chains}. This yields classifications of posets with a given $\mathfrak{s}_P$ not divisible by 24. As a special case, we have the following proposition.

\begin{prop}
    We have $\mathfrak{s}_P = 1$ (i.e., the group action is simply transitive) if and only if $P$ is an ordinal sum of antichains. 
\end{prop}
In fact, the $\mathcal{BK}$ group of an ordinal sum of antichains $A_{i_1} \oplus A_{i_2} \oplus \dots \oplus A_{i_n}$ is the product of symmetric groups $\mathfrak{S}_{i_1} \times \mathfrak{S}_{i_2} \times \dots \times \mathfrak{S}_{i_n}$.

We now define a poset parameter that gives us useful information about the stabilizer, including a lower bound for $\mathfrak{s}_P$.
\begin{definition}
    For a poset $P$ and $\ell \in \linext(P)$, define
    \[
        c(P,\ell) := |\{i \in [1,|P|-1] \mid \ell^{-1}(i) < \ell^{-1}(i+1)\}|.
    \]
    Then, the \emph{comparability} of $P$ is
    \[
        c(P) := \max_\ell c(P,\ell).
    \]
    In other words, $c(P)$ is the maximum number of $t_i$ such that $t_i \in \Stab(\ell)$.
\end{definition}

\begin{definition}
    Let $G$ be a group. We say a set $S \subseteq G$ is \emph{independent} if $s \notin \langle S \setminus \{s\} \rangle$ for all $s \in S$. Then, let $m(G)$ be the maximum size of an independent set of involutions in $G$.
\end{definition}

\begin{prop} \label{prop-m-log}
    We have $m(G) \le \log_2(|G|)$.
\end{prop}
\begin{proof}
    Given an independent set $\{g_1, \dots, g_{m(G)}\} \subseteq G$, we have
    \[
        \langle g_1 \rangle < \langle g_1, g_2 \rangle < \dots < \langle g_1, g_2, \dots, g_{m(G)} \rangle,
    \]
    where each subgroup is at least twice as large as the previous one by Lagrange's theorem.
\end{proof}

\begin{prop} \label{prop-c-stab}
    If $P$ is indecomposable, then $c(P) \le m(\Stab(\ell))$.
\end{prop}
\begin{proof}
    By Corollary~\ref{ti-indep}, $\ell$ has at least $c(P)$ independent involutions in $\Stab(\ell)$.
\end{proof}

\begin{cor}
    If $P$ is indecomposable, then $\mathfrak{s}_P \ge 2^{c(P)}$.
\end{cor}
\begin{proof}
    This immediate from Propositions~\ref{prop-m-log} and \ref{prop-c-stab}.
\end{proof}

\begin{prop} \label{prop-c-h-w}
    We have $h(P) - 1 \le c(P) \le |P| - w(P)$, where $h(P)$ is the height of $P$ and $w(P)$ is the width of $P$. As a consequence, $\mathfrak{s}_P \ge 2^{h(P) - 1}$.
\end{prop}
\begin{proof}
    For the first inequality, let $P'$ be a subposet of $P$ obtained by removing a maximum chain $C$ from $P$, so $|P'| = |P| - h(P)$. We construct $\ell \in \linext(P)$ as follows. At each step, if label $i$ can be assigned to the least element in $C$ without a label, then do so; otherwise, assign the label $i$ to any other element for which it is possible. If $\ell^{-1}(i), \ell^{-1}(i+1) \in C$, then $\ell^{-1}(i) < \ell^{-1}(i+1)$. Further, if $\ell^{-1}(i) \in P'$ and $\ell^{-1}(i+1) \in C$, then $\ell^{-1}(i) < \ell^{-1}(i+1)$. It follows that $h(P) - 1 \le c(P)$.
    
    For the second inequality, let $A$ be a maximum antichain in $P$. Let $\ell \in \linext(P)$. Then, given $x, y \in A$ with $\ell(x) < \ell(y)$, there must be some $i \in [\ell(x), \ell(y) - 1]$ such that $\ell^{-1}(i)$ is incomparable to $\ell^{-1}(i+1)$. It follows that $c(P) \le |P| - w(P)$.
\end{proof}


\bibliography{bibliography}
\bibliographystyle{alpha}

\end{document}